\documentclass[12pt,reqno]{amsart}

\usepackage[colorlinks,linkcolor={blue},
citecolor={blue},urlcolor={blue}]{hyperref}

\usepackage{mathrsfs}
\usepackage{bbm}
\usepackage{amsfonts}
\usepackage{}
\usepackage{pifont}
\usepackage[shortlabels]{enumitem}
\usepackage{empheq}
\usepackage{amsfonts}
\usepackage{amsmath, amsrefs, amsthm, amssymb}
\usepackage[active]{srcltx}		
\usepackage{pdfsync}					
\usepackage{graphicx}
\usepackage [latin1]{inputenc}
\usepackage{bbm}
\usepackage{amsfonts}
\usepackage{amsthm}
\usepackage{graphicx}
\usepackage{framed}
\usepackage{multicol,multienum}
\usepackage{graphicx} 
\usepackage{booktabs} 
\usepackage{mathrsfs}
\usepackage{tcolorbox}
\usepackage{color}
\usepackage{xcolor}

\newcommand{\wuhao}{\fontsize{10.5pt}{\baselineskip}\selectfont}

\newtheorem{theorem}{Theorem}[section]
\newtheorem{proposition}[theorem]{Proposition}
\newtheorem{lemma}[theorem]{Lemma}
\newtheorem{corollary}[theorem]{Corollary}

\theoremstyle{definition}

\newtheorem{remark}[theorem]{Remark}

\numberwithin{equation}{section}

\begin{document}

\title [The LLBar equation]{\wuhao Global well-posedness for the Landau-Lifshitz-Baryakhtar equation in $\mathbb{R}^3$}

\author{Fan Xu}
\address{School of Mathematics and Statistics, Hubei Key Laboratory of Engineering Modeling  and Scientific Computing, Huazhong University of Science and Technology,  Wuhan 430074, Hubei, P.R. China.}
\email{d202280019@hust.edu.cn (F. Xu)}

\author{Bin Liu}
\address{School of Mathematics and Statistics, Hubei Key Laboratory of Engineering Modeling  and Scientific Computing, Huazhong University of Science and Technology,  Wuhan 430074, Hubei, P.R. China.}
\email{binliu@mail.hust.edu.cn (B. Liu)}


\begin{abstract}
This paper establishes the global well-posedness of the Landau-Lifshitz-Baryakhtar (LLBar) equation in the whole space $\mathbb{R}^3$. The study first demonstrates the existence and uniqueness of global strong solutions using the weak compactness approach. Furthermore, the existence and uniqueness of classical solutions, as well as arbitrary smooth solutions, are derived through a bootstrap argument. The proofs for the existence of these three types of global solutions are based on Friedrichs mollifier approximation and energy estimates, with the structure of the LLBar equation playing a crucial role in the derivation of the results.
\end{abstract}

\maketitle
\section{Introduction}\label{sec1}

\subsection{Statement of the problem}
The origins of ferromagnetism theory can be traced back to Weiss's seminal work in \cite{brown1963micromagnetics,weiss1907hypothese}, where he introduced the molecular field theory to explain the macroscopic magnetization of ferromagnetic materials. Subsequently, Landau and Lifshitz \cite{landau1935theory} formulated a dynamical theory for the magnetization of ferromagnetic materials, leading to the introduction of the celebrated Landau-Lifshitz (LL) equation. This equation describes the evolution of the spin magnetic moment in ferromagnetic systems, especially in the context of precession and nonlinear dynamics under external magnetic fields. Later, Gilbert \cite{gilbert1955lagrangian} made a significant contribution by extending the theory through the Landau-Lifshitz-Gilbert (LLG) equation, which incorporates dissipation effects through a damping term. The LLG equation provides a more detailed description of the evolution of the spin magnetic moment, particularly in the context of precession and energy dissipation under an external magnetic field. Moreover, Garanin \cite{garanin1997fokker} developed a thermodynamically consistent framework for spin dynamics and derived the Landau-Lifshitz-Bloch (LLB) equation. The LLB equation improves upon the LL and LLG models by offering a more accurate description of magnetic dynamics at high temperatures, especially near the Curie temperature, where thermal fluctuations become significant. Over the past several decades, the existence, uniqueness, and regularity of solutions to both the LLG and LLB equations have been the subject of extensive mathematical investigation, with notable references including \cite{alouges1992global,carbou2001regular1,carbou2001regular,feischl2017existence,gutierrez2019cauchy,he2023landau,le2016weak,li2021weak,li2021smooth,lin2015global,peng2022strong,pu2022global} and others cited therein. To address experimental observations and microscopic calculations, including the non-local damping effects in magnetic metals and crystals or the unexpectedly high attenuation of short-wavelength magnons, Baryakhtar \cite{baryakhtar1984phenomenological,baryakhtar2013phenomenological,baryakhtar1997soliton} extended both the LLG and LLB models. He introduced the Landau-Lifshitz-Baryakhtar (LLBar) equation, a fourth-order nonlinear parabolic equation, to better capture these phenomena. The LLBar equation incorporates higher-order terms that account for non-local damping effects and the attenuation of high-frequency excitations, providing a more comprehensive description of the magnetic behavior in materials with complex dynamical properties. The dynamic evolution of its solutions is described by
\begin{equation}\label{sys0}
\left\{
\begin{aligned}
&\mathbf{u}_t=\lambda_r\mathbf{H}_{\textrm{eff}}-\lambda_e\Delta\mathbf{H}_{\textrm{eff}}-\gamma\mathbf{u}\times \mathbf{H}_{\textrm{eff}},\\
&\mathbf{H}_{\textrm{eff}}=\Delta \mathbf{u}+\frac{1}{2\chi}(1-|\mathbf{u}|^2)\mathbf{u},
\end{aligned}
\right.
\end{equation}
where the unknown quantity $\mathbf{u}(t,x)\in\mathbb{R}^3$ denotes the magnetization vector. The positive constants $\lambda_r$, $\lambda_e$, and $\gamma$ are the relativistic damping constant, the exchange damping constant, and the electron gyromagnetic ratio, respectively. The positive constant $\chi$ is the magnetic susceptibility of the material. Without loss of generality, we assume that $\chi=\frac{1}{4}$, $\lambda_r=\lambda_e=\gamma=1$ in this paper. The notation $\mathbf{H}_{\textrm{eff}}$ denotes the effective field, which consists of the external magnetic field, the demagnetizing field and some quantum mechanical effects, etc.  For more details on the background of the LLG,  LLB and LLBar equations, we refer to \cite{atxitia2016fundamentals,dvornik2013micromagnetic,dvornik2014thermodynamically,guo2008landau,mayergoyz2009nonlinear,soenjaya2023global,vogler2014landau} and the references therein.

The mathematical analysis of the LLBar equation has recently been undertaken by Soenjaya and Tran \cite{soenjaya2023global}, who proved the existence and uniqueness of global weak (strong) solutions in bounded domains with smooth boundaries using the classical Galerkin finite-dimensional approximation method. Building upon the analysis in bounded domains, we now extend the problem to the three-dimensional whole space, transitioning to the Cauchy problem. More specifically, in this paper, we consider the Cauchy problem for the following LLBar equation:
\begin{equation}\label{sys1}
\left\{
\begin{aligned}
&\mathbf{u}_t=-\Delta^2\mathbf{u}-\Delta\mathbf{u}+2(1-|\mathbf{u}|^2)\mathbf{u}+2\Delta(|\mathbf{u}|^2\mathbf{u})-\mathbf{u}\times\Delta \mathbf{u},&\textrm{in}~\mathbb{R}_+\times\mathbb{R}^3,\\
&\mathbf{u}(0)=\mathbf{u}_0,&\textrm{in}~\mathbb{R}^3.
\end{aligned}
\right.
\end{equation}
We shall establish the existence and uniqueness of a global strong solution to system \eqref{sys1}, along with the existence of classical and arbitrarily smooth solutions under appropriate regularity conditions on the initial data. The proofs of these results are framed within a unified approach, which combines Friedrichs mollifier approximation with energy estimates.

\subsection{Main results}
In this paper, we adopt the following conventions: The inequality $A \lesssim_{a,b,\cdots} B$ means that there exists a positive constant $C$ depending only on $a, b, \cdots$ such that $A \leq C B$, while $A \asymp_{a,b,\cdots} B$ indicates that there exist two positive constants $C \leq C'$ depending only on $a, b, \cdots$ such that $C B \leq A \leq C' B$.

Let $\mathbb{L}^p := L^p(\mathbb{R}^3; \mathbb{R}^3)$ be the space of $\mathbb{R}^3$-valued functions that are $p$-th power Lebesgue integrable. Given a Banach space $X$, the symbol $X' := \mathcal{L}(X; \mathbb{R})$ denotes its dual space. Let $\mathcal{C}_0^{\infty}(\mathbb{R}^3)$ denote the space of all $\mathbb{R}$-valued functions of class $\mathcal{C}^{\infty}$ with compact support. Let $\Lambda^s := (I - \Delta)^{\frac{s}{2}}$ denote the Bessel potential operator~\cite{bahouri2011fourier}. Define the Sobolev space $H^s(\mathbb{R}^3)$ as the Hilbert space endowed with the norm
\begin{equation*}
\|f\|_{H^s} = \|\Lambda^s f\|_{L^2} = \left[ \int_{\mathbb{R}^3} (1 + |\xi|^2)^s |\hat{f}(\xi)|^2\, \mathrm{d}\xi \right]^{\frac{1}{2}},
\end{equation*}
where $\hat{f}$ denotes the Fourier transform of a tempered distribution $f$. Let $\mathbb{H}^s := H^s(\mathbb{R}^3; \mathbb{R}^3)$.

The main results can now be stated as follows.

\begin{theorem}\label{the1}
Fix arbitrary $T>0$.
\begin{itemize}
\item[\textbf{(1)}] If $\mathbf{u}_0 \in \mathbb{H}^2$, then system \eqref{sys1} admits a unique global strong solution
$\mathbf{u} \in \mathcal{C}([0,T]; \mathbb{H}^2) \cap L^2(0,T; \mathbb{H}^4)$,
such that for any $t \in [0,T]$ and $\phi \in \mathbb{L}^2$,
\begin{equation}\label{equ2}
\begin{split}
(\mathbf{u}(t), \phi)_{\mathbb{L}^2} &= (\mathbf{u}_0, \phi)_{\mathbb{L}^2}
- \int_0^t (\Delta^2 \mathbf{u}, \phi)_{\mathbb{L}^2} \,\mathrm{d}s
- \int_0^t (\Delta \mathbf{u}, \phi)_{\mathbb{L}^2} \,\mathrm{d}s \\
&\quad + 2 \int_0^t ((1 - |\mathbf{u}|^2)\mathbf{u}, \phi)_{\mathbb{L}^2} \,\mathrm{d}s
+ 2 \int_0^t (\Delta (|\mathbf{u}|^2 \mathbf{u}), \phi)_{\mathbb{L}^2} \,\mathrm{d}s- \int_0^t (\mathbf{u} \times \Delta \mathbf{u}, \phi)_{\mathbb{L}^2} \,\mathrm{d}s.
\end{split}
\end{equation}

\item[\textbf{(2)}] If $\mathbf{u}_0 \in \mathbb{H}^6$, then system \eqref{sys1} admits a unique global classical solution
$\mathbf{u} \in \mathcal{C}([0,T]; \mathcal{C}^4(\mathbb{R}^3)) \cap \mathcal{C}^1([0,T]; \mathcal{C}(\mathbb{R}^3))$.

\item[\textbf{(3)}] If $\mathbf{u}_0 \in \mathbb{H}^{6 + 4m}$ for some $m \in \mathbb{N}^+$, then system \eqref{sys1} admits a unique global smooth solution
$\mathbf{u} \in \bigcap_{k=0}^{m+1} \mathcal{C}^k([0,T]; \mathcal{C}^{4(m+1 - k)}(\mathbb{R}^3))$.
\end{itemize}
\end{theorem}

\begin{remark}
To the best of our knowledge, Theorem~\ref{the1} provides the first mathematical result on the Cauchy problem for the LLBar equation. In particular, conclusions \textbf{(2)} and \textbf{(3)} of Theorem~\ref{the1} establish the global existence and uniqueness of classical and arbitrarily smooth solutions under large initial data, which has not been addressed in previous work on initial-boundary value problem~\cite{soenjaya2023global}. The methods used in this paper differ slightly from those in the existing literature on the Landau-Lifshitz equations.
\end{remark}

\subsection{Organization}
This paper is organized as follows. In Section~\ref{sec2}, we establish the existence and uniqueness of a global smooth solution to the approximated system~\eqref{Mod-1}. In Section~\ref{sec3}, we derive uniform bounded estimates (independent of $\epsilon$) for the solutions $\mathbf{u}^{\epsilon}$ of the approximated system~\eqref{Mod-1}. In Section~\ref{sec4}, we analyze the convergence of the approximate solutions $\mathbf{u}^{\epsilon}$ to obtain local strong and local smooth solutions for the original system~\eqref{sys1}. Finally, in Section~\ref{sec5}, we leverage the intrinsic properties of the equation to extend the local solutions to global ones on the time interval $[0,T]$, and prove the uniqueness of the global solution, thus completing the proof of Theorem~\ref{the1}.

\section{Global solvability of the modified system}\label{sec2}
Let us first define the Friedrichs mollifier. Given any radial function
\begin{equation*}
\begin{split}
\rho(|x|)\in\mathcal{C}_0^{\infty}(\mathbb{R}^3),~\rho\geq0,~\int_{\mathbb{R}^3}\rho\,\mathrm{d}x=1,
\end{split}
\end{equation*}
define the mollification $J_{\epsilon}f$ of functions $f\in L^p(\mathbb{R}^3),~p\in[1,\infty]$, by
\begin{equation*}
\begin{split}
(J_{\epsilon}f)(x):=(\rho*f)(x):=\epsilon^{-3}\int_{\mathbb{R}^3}\rho(\frac{x-y}{\epsilon})f(y)\,\mathrm{d}y,~\epsilon>0.
\end{split}
\end{equation*}
The regularizing operator $J_{\epsilon}$ has the following well-known properties (cf. \cite{taylor1996partial}).
\begin{lemma}\label{lem0} Let $J_{\epsilon}$ be the mollifier as defined above.
\begin{itemize}
\item [{(\textbf{i})}] Mollifier $J_{\epsilon}$ commutes with distribution derivatives.
\item [(\textbf{ii})] For all $f\in \mathcal{C}(\mathbb{R}^3)$, $J_{\epsilon}f\rightarrow f$ uniformly on any compact set $\mathcal{O}\subset \mathbb{R}^3$ and
\begin{equation*}
\begin{split}
\|J_{\epsilon}f\|_{L^{\infty}}\leq \|f\|_{L^{\infty}}.
\end{split}
\end{equation*}
\item [{(\textbf{iii})}] For all $f,~g\in L^2(\mathbb{R}^3)$,
\begin{equation*}
\begin{split}
(J_{\epsilon}f,g)_{L^2}=(f,J_{\epsilon}g)_{L^2}.
\end{split}
\end{equation*}
\item [{(\textbf{iv})}] For all $f\in H^s(\mathbb{R}^3)$,~$s\geq0$, and $g\in H^l(\mathbb{R}^3)$,~$l\geq1$,
\begin{equation*}
\begin{split}
\lim_{\epsilon\searrow0}\|J_{\epsilon}f-f\|_{H^s}=0,~\|J_{\epsilon}g-g\|_{H^{l-1}}\leq C\epsilon\|g\|_{H^l}.
\end{split}
\end{equation*}
\item [{(\textbf{v})}] For all $f\in H^m(\mathbb{R}^3)$,~$k\in \mathbb{Z}^{+}\cup\{0\}$,~and $\epsilon>0$,
\begin{equation*}
\begin{split}
\|J_{\epsilon}f\|_{H^{m+k}}\leq C_{m,k}\epsilon^{-k}\|f\|_{H^m},~\|J_{\epsilon}D^kf\|_{L^{\infty}}\leq C_{k}\epsilon^{-\frac{3}{2}-k}\|f\|_{L^2}.
\end{split}
\end{equation*}
\end{itemize}
\end{lemma}

We now consider the following modified system:
\begin{equation}\label{Mod-1}
\left\{
\begin{aligned}
&\mathbf{u}_t^{\epsilon}=-J_{\epsilon}[\Delta^2(J_{\epsilon}\mathbf{u}^{\epsilon})]-J_{\epsilon}[\Delta (J_{\epsilon}\mathbf{u}^{\epsilon})]+2J_{\epsilon}[(1-|J_{\epsilon}\mathbf{u}^{\epsilon}|^2)J_{\epsilon}\mathbf{u}^{\epsilon}]+2J_{\epsilon}[\Delta(|J_{\epsilon}\mathbf{u}^{\epsilon}|^2J_{\epsilon}\mathbf{u}^{\epsilon})]\\
&-J_{\epsilon}[J_{\epsilon}\mathbf{u}^{\epsilon}\times \Delta (J_{\epsilon}\mathbf{u}^{\epsilon})],\\
&\mathbf{u}_0^{\epsilon}=J_{\epsilon}\mathbf{u}_0.
\end{aligned}
\right.
\end{equation}
System \eqref{Mod-1} can be expressed as an ODE in the work space $\mathbb{H}^s$:
\begin{equation}\label{Mod-2}
\left\{
\begin{aligned}
&\frac{\mathrm{d}}{\mathrm{d}t}  \mathbf{u}^{\epsilon}=F^{\epsilon}(\mathbf{u}^{\epsilon}),\\
&\mathbf{u}_0^{\epsilon}=J_{\epsilon}\mathbf{u}_0,
\end{aligned}
\right.
\end{equation}
where
\begin{equation}\label{def1}
\begin{split}
&F^{\epsilon}(\mathbf{u}^{\epsilon}):=-J_{\epsilon}[\Delta^2(J_{\epsilon}\mathbf{u}^{\epsilon})]-J_{\epsilon}[\Delta (J_{\epsilon}\mathbf{u}^{\epsilon})]+2J_{\epsilon}[(1-|J_{\epsilon}\mathbf{u}^{\epsilon}|^2)J_{\epsilon}\mathbf{u}^{\epsilon}]+2J_{\epsilon}[\Delta(|J_{\epsilon}\mathbf{u}^{\epsilon}|^2J_{\epsilon}\mathbf{u}^{\epsilon})]\\
&-J_{\epsilon}[J_{\epsilon}\mathbf{u}^{\epsilon}\times \Delta (J_{\epsilon}\mathbf{u}^{\epsilon})].
\end{split}
\end{equation}
The following lemma demonstrates the existence and uniqueness of a global smooth solution for the regularized system \eqref{Mod-2}.

\begin{lemma}\label{lem1} Let integer $s\geq10$, $1>\epsilon>0$, then system \eqref{Mod-2} admits a unique global solution $\mathbf{u}\in \mathcal {C}^1([0,T];\mathbb{H}^s)$.
\end{lemma}

\begin{proof}[\emph{\textbf{Proof}}]
For fixed $\|\mathbf{u}^{\epsilon}\|_{\mathbb{H}^s}<M$ and $\|\mathbf{v}^{\epsilon}\|_{\mathbb{H}^s}<M$, by using Lemma \ref{lem0} and Sobolev embedding $H^2(\mathbb{R}^3)\hookrightarrow L^{\infty}(\mathbb{R}^3)$, we have
\begin{equation}\label{lem1-1}
\begin{split}
&\|F^{\epsilon}(\mathbf{u}^{\epsilon})-F^{\epsilon}(\mathbf{v}^{\epsilon})\|_{\mathbb{H}^s}\lesssim \|J_{\epsilon}^2\Delta^2(\mathbf{u}^{\epsilon}-\mathbf{v}^{\epsilon})\|_{\mathbb{H}^s}+\|J_{\epsilon}^2\Delta(\mathbf{u}^{\epsilon}-\mathbf{v}^{\epsilon})\|_{\mathbb{H}^s}+\|J_{\epsilon}^2(\mathbf{u}^{\epsilon}-\mathbf{v}^{\epsilon})\|_{\mathbb{H}^s}\\
&+\|J_{\epsilon}(|J_{\epsilon}\mathbf{u}^{\epsilon}|^2J_{\epsilon}\mathbf{u}^{\epsilon}-|J_{\epsilon}\mathbf{v}^{\epsilon}|^2J_{\epsilon}\mathbf{v}^{\epsilon})\|_{\mathbb{H}^s}+\|J_{\epsilon}\Delta(|J_{\epsilon}\mathbf{u}^{\epsilon}|^2J_{\epsilon}\mathbf{u}^{\epsilon}-|J_{\epsilon}\mathbf{v}^{\epsilon}|^2J_{\epsilon}\mathbf{v}^{\epsilon})\|_{\mathbb{H}^s}\\
&+\|J_{\epsilon}[J_{\epsilon}\mathbf{u}^{\epsilon}\times \Delta (J_{\epsilon}\mathbf{u}^{\epsilon})-J_{\epsilon}\mathbf{v}^{\epsilon}\times \Delta (J_{\epsilon}\mathbf{v}^{\epsilon})]\|_{\mathbb{H}^s}\\
&\lesssim_s\epsilon^{-4}\|\mathbf{u}^{\epsilon}-\mathbf{v}^{\epsilon}\|_{\mathbb{H}^s}+\epsilon^{-2}\|\mathbf{u}^{\epsilon}-\mathbf{v}^{\epsilon}\|_{\mathbb{H}^s}+\|\mathbf{u}^{\epsilon}-\mathbf{v}^{\epsilon}\|_{\mathbb{H}^s}+\epsilon^{-s}\||J_{\epsilon}\mathbf{u}^{\epsilon}|^2J_{\epsilon}\mathbf{u}^{\epsilon}-|J_{\epsilon}\mathbf{v}^{\epsilon}|^2J_{\epsilon}\mathbf{v}^{\epsilon}\|_{\mathbb{L}^2}\\
&+\epsilon^{-(s+2)}\||J_{\epsilon}\mathbf{u}^{\epsilon}|^2J_{\epsilon}\mathbf{u}^{\epsilon}-|J_{\epsilon}\mathbf{v}^{\epsilon}|^2J_{\epsilon}\mathbf{v}^{\epsilon}\|_{\mathbb{L}^2}+\epsilon^{-s}\|J_{\epsilon}\mathbf{u}^{\epsilon}\times \Delta (J_{\epsilon}\mathbf{u}^{\epsilon})-J_{\epsilon}\mathbf{v}^{\epsilon}\times \Delta (J_{\epsilon}\mathbf{v}^{\epsilon})\|_{\mathbb{L}^2}\\
&\lesssim_{s,\epsilon}\|\mathbf{u}^{\epsilon}-\mathbf{v}^{\epsilon}\|_{\mathbb{H}^s}+\|J_{\epsilon}\mathbf{u}^{\epsilon}\|_{\mathbb{L}^{\infty}}(\|J_{\epsilon}\mathbf{u}^{\epsilon}\|_{\mathbb{L}^{\infty}}+\|J_{\epsilon}\mathbf{v}^{\epsilon}\|_{\mathbb{L}^{\infty}})\|\mathbf{u}^{\epsilon}-\mathbf{v}^{\epsilon}\|_{\mathbb{L}^2}+\|J_{\epsilon}\mathbf{v}^{\epsilon}\|_{\mathbb{L}^{\infty}}^2\|\mathbf{u}^{\epsilon}-\mathbf{v}^{\epsilon}\|_{\mathbb{L}^2}\\
&+\|J_{\epsilon}\Delta\mathbf{u}^{\epsilon}\|_{\mathbb{L}^{\infty}}^2\|\mathbf{u}^{\epsilon}-\mathbf{v}^{\epsilon}\|_{\mathbb{L}^2}+\|J_{\epsilon}\mathbf{v}^{\epsilon}\|_{\mathbb{L}^{\infty}}^2\|J_{\epsilon}\Delta(\mathbf{u}^{\epsilon}-\mathbf{v}^{\epsilon})\|_{\mathbb{L}^2}\\
&\lesssim_{s,\epsilon}(1+\|\mathbf{u}^{\epsilon}\|_{\mathbb{L}^2}^2+\|\mathbf{v}^{\epsilon}\|_{\mathbb{L}^2}^2)\|\mathbf{u}^{\epsilon}-\mathbf{v}^{\epsilon}\|_{\mathbb{H}^s}\lesssim_{M,s,\epsilon}\|\mathbf{u}^{\epsilon}-\mathbf{v}^{\epsilon}\|_{\mathbb{H}^s}.
\end{split}
\end{equation}
Thus, the mapping $F^{\epsilon}(\cdot):\mathbb{H}^s\mapsto \mathbb{H}^s$ is locally Lipschitz continuous. This allows us to apply the Picard theorem on Banach spaces to prove that system \eqref{Mod-2} admits a unique solution $\mathbf{u}^{\epsilon}\in \mathcal{C}^1([0,T_{\epsilon});\mathbb{H}^s)$ where $T_{\epsilon}$ is the maximal existence time.

Next, we shall prove that $T_{\epsilon}=\infty$. Taking the $\mathbb{L}^2$ inner product of equation \eqref{Mod-1} with $\mathbf{u}^{\epsilon}$ and using integration by parts, it follows that
\begin{equation}\label{lem1-2}
\begin{split}
&\frac{1}{2}\frac{\mathrm{d}}{\mathrm{d}t}\|\mathbf{u}^{\epsilon}\|_{\mathbb{L}^2}^2+\|\Delta J_{\epsilon}\mathbf{u}^{\epsilon}\|_{\mathbb{L}^2}^2+2\|J_{\epsilon}\mathbf{u}^{\epsilon}\|_{\mathbb{L}^4}^4+4\| J_{\epsilon}\mathbf{u}^{\epsilon}\cdot\nabla J_{\epsilon}\mathbf{u}^{\epsilon}\|_{\mathbb{L}^2}^2+2\||J_{\epsilon}\mathbf{u}^{\epsilon}||\nabla J_{\epsilon}\mathbf{u}^{\epsilon}|\|_{\mathbb{L}^2}^2\\
&=\|\nabla J_{\epsilon}\mathbf{u}^{\epsilon}\|_{\mathbb{L}^2}^2+2\|J_{\epsilon}\mathbf{u}^{\epsilon}\|_{\mathbb{L}^2}^2\leq \frac{1}{2}\|\Delta J_{\epsilon}\mathbf{u}^{\epsilon}\|_{\mathbb{L}^2}^2+C\|\mathbf{u}^{\epsilon}\|_{\mathbb{L}^2}^2.
\end{split}
\end{equation}
By using the Gronwall lemma to \eqref{lem1-2}, we conclude that
\begin{equation}\label{lem1-3}
\begin{split}
&\sup_{t\in[0,T]}\|\mathbf{u}^{\epsilon}(t)\|_{\mathbb{L}^2}^2\lesssim_{\mathbf{u}_0,T}1.
\end{split}
\end{equation}
Moreover, we note that relation \eqref{lem1-1} with $\mathbf{v}^{\epsilon}=0$ gives the bound
\begin{equation*}
\begin{split}
&\frac{\mathrm{d}}{\mathrm{d}t}\|\mathbf{u}^{\epsilon}\|_{\mathbb{H}^s}\lesssim_{s,\epsilon}(1+\|\mathbf{u}^{\epsilon}\|_{\mathbb{L}^2}^2)\|\mathbf{u}^{\epsilon}\|_{\mathbb{H}^s},
\end{split}
\end{equation*}
which together with \eqref{lem1-3} and Gronwall's lemma implies that
\begin{equation*}
\begin{split}
&\sup_{t\in[0,T]}\|\mathbf{u}^{\epsilon}(t)\|_{\mathbb{H}^s}\lesssim_{s,\epsilon,\mathbf{u}_0,T}1.
\end{split}
\end{equation*}
Thus invoking the continuation property of ODEs on Banach spaces (cf. \cite[Theorem 3.3]{majda2002vorticity}), we conclude that the solution exists globally in time. The proof is thus completed.
\end{proof}

\section{Uniformly bounded estimate}\label{sec3}
This section is devoted to establishing a uniform bound (independent of $\epsilon$) for the approximate solution $\mathbf{u}^\epsilon$ in the relevant energy space, as a preparation for proving the results stated in Theorem~\ref{the1}.

We begin by considering the case where the initial value belongs to $\mathbb{H}^2$.
\begin{lemma}\label{lem3} Let $\mathbf{u}_0\in\mathbb{H}^2$. Then there exists $T^*=T^*(\|\nabla\mathbf{u}_0\|_{\mathbb{L}^2})>0$ such that for any $\epsilon\in(0,1)$, $T'\in(0,T^*)$ and $t\in[0,T']$,
\begin{align}
&\|\mathbf{u}^{\epsilon}(t)\|_{\mathbb{H}^2}^{2}+\int_0^t\|J_{\epsilon}\mathbf{u}^{\epsilon}\|_{\mathbb{H}^4}^2\,\mathrm{d}s\lesssim_{\|\mathbf{u}_0\|_{\mathbb{H}^2},T'}1,\label{lem3-1}\\
&\|\mathbf{u}^{\epsilon}_t\|_{L^2(0,T';\mathbb{L}^{2})}^2\lesssim_{\|\mathbf{u}_0\|_{\mathbb{H}^2},T'}1.\label{lem3-2}
\end{align}
\end{lemma}
\begin{proof}[\emph{\textbf{Proof}}] Similar to the proofs of \eqref{lem1-2} and \eqref{lem1-3}, it is straightforward to obtain that for all $t\in[0,T]$,
\begin{equation}\label{lem3-3}
\begin{split}
&\|\mathbf{u}^{\epsilon}\|_{\mathbb{L}^2}^2+\int_0^t\|J_{\epsilon}\mathbf{u}^{\epsilon}\|_{\mathbb{H}^2}^2\lesssim_{\|\mathbf{u}_0\|_{\mathbb{L}^2},T}1.
\end{split}
\end{equation}
Next, multiplying both sides of equation \eqref{Mod-1} by $-\Delta\mathbf{u}^{\epsilon}$ and using integration by parts, it follows that
\begin{equation}\label{lem3-4}
\begin{split}
&\frac{1}{2}\frac{\mathrm{d}}{\mathrm{d}t}\|\nabla\mathbf{u}^{\epsilon}\|_{\mathbb{L}^2}^2+\|\nabla\Delta J_{\epsilon}\mathbf{u}^{\epsilon}\|_{\mathbb{L}^2}^2+4\| J_{\epsilon}\mathbf{u}^{\epsilon}\cdot\nabla J_{\epsilon}\mathbf{u}^{\epsilon}\|_{\mathbb{L}^2}^2+2\||J_{\epsilon}\mathbf{u}^{\epsilon}||\nabla J_{\epsilon}\mathbf{u}^{\epsilon}|\|_{\mathbb{L}^2}^2\\
&=\|\Delta J_{\epsilon}\mathbf{u}^{\epsilon}\|_{\mathbb{L}^2}^2+2\|\nabla J_{\epsilon}\mathbf{u}^{\epsilon}\|_{\mathbb{L}^2}^2-2(\Delta(|J_{\epsilon}\mathbf{u}^{\epsilon}|^2J_{\epsilon}\mathbf{u}^{\epsilon}),\Delta J_{\epsilon}\mathbf{u}^{\epsilon})_{\mathbb{L}^2}.
\end{split}
\end{equation}
Here, regarding the above equation, we used the fact that $$(J_{\epsilon}(J_{\epsilon}\mathbf{u}^{\epsilon}\times \Delta J_{\epsilon}\mathbf{u}^{\epsilon}),\Delta \mathbf{u}^{\epsilon})_{\mathbb{L}^2}=(J_{\epsilon}\mathbf{u}^{\epsilon}\times \Delta J_{\epsilon}\mathbf{u}^{\epsilon},\Delta J_{\epsilon}\mathbf{u}^{\epsilon})_{\mathbb{L}^2}=0.$$
Through direct computation, we have
\begin{equation*}
\begin{split}
&2(\Delta(|J_{\epsilon}\mathbf{u}^{\epsilon}|^2J_{\epsilon}\mathbf{u}^{\epsilon}),\Delta J_{\epsilon} \mathbf{u}^{\epsilon})_{\mathbb{L}^2}=4\| J_{\epsilon}\mathbf{u}^{\epsilon}\cdot\Delta J_{\epsilon}\mathbf{u}^{\epsilon}\|_{\mathbb{L}^2}^2+2\||J_{\epsilon}\mathbf{u}^{\epsilon}||\Delta J_{\epsilon}\mathbf{u}^{\epsilon}|\|_{\mathbb{L}^2}^2\\
&+8(\nabla J_{\epsilon}\mathbf{u}^{\epsilon}(J_{\epsilon}\mathbf{u}^{\epsilon}\cdot\nabla J_{\epsilon}\mathbf{u}^{\epsilon})^{\top},\Delta J_{\epsilon}\mathbf{u}^{\epsilon})_{\mathbb{L}^2}+4(|\nabla J_{\epsilon}\mathbf{u}^{\epsilon}|^2J_{\epsilon}\mathbf{u}^{\epsilon},\Delta J_{\epsilon}\mathbf{u}^{\epsilon})_{\mathbb{L}^2}.
\end{split}
\end{equation*}
Combining this with \eqref{lem3-4} and using the Gagliardo-Nirenberg interpolation inequality \cite{gagliardo1959ulteriori,gilbarg1977elliptic}
\begin{equation*}
\begin{split}
&\|\nabla f\|_{L^4(\mathbb{R}^3)}^4\lesssim\|\nabla\Delta f\|_{L^2(\mathbb{R}^3)}^{\frac{3}{2}}\|\nabla f\|_{L^2(\mathbb{R}^3)}^{\frac{5}{2}},~f\in H^3(\mathbb{R}^3),
\end{split}
\end{equation*}
it follows from \eqref{lem3-4} that
\begin{equation}\label{lem3-5}
\begin{split}
&\frac{\mathrm{d}}{\mathrm{d}t}\|\nabla\mathbf{u}^{\epsilon}\|_{\mathbb{L}^2}^2+\|\nabla\Delta J_{\epsilon}\mathbf{u}^{\epsilon}\|_{\mathbb{L}^2}^2+\||J_{\epsilon}\mathbf{u}^{\epsilon}||\Delta J_{\epsilon}\mathbf{u}^{\epsilon}|\|_{\mathbb{L}^2}^2\\
&\lesssim\|\Delta J_{\epsilon}\mathbf{u}^{\epsilon}\|_{\mathbb{L}^2}^2+\|\nabla J_{\epsilon}\mathbf{u}^{\epsilon}\|_{\mathbb{L}^2}^2+\||J_{\epsilon}\mathbf{u}^{\epsilon}||\Delta J_{\epsilon}\mathbf{u}^{\epsilon}|\|_{\mathbb{L}^2}\|\nabla J_{\epsilon}\mathbf{u}^{\epsilon}\|_{\mathbb{L}^4}^2\\
&\leq\varepsilon\|\nabla\Delta J_{\epsilon}\mathbf{u}^{\epsilon}\|_{\mathbb{L}^2}^2+\varepsilon\||J_{\epsilon}\mathbf{u}^{\epsilon}||\Delta J_{\epsilon}\mathbf{u}^{\epsilon}|\|_{\mathbb{L}^2}^2+C_{\varepsilon}\|\nabla J_{\epsilon}\mathbf{u}^{\epsilon}\|_{\mathbb{L}^2}^2+C_{\varepsilon}\|\nabla J_{\epsilon}\mathbf{u}^{\epsilon}\|_{\mathbb{L}^4}^4\\
&\leq\varepsilon\|\nabla\Delta J_{\epsilon}\mathbf{u}^{\epsilon}\|_{\mathbb{L}^2}^2+\varepsilon\||J_{\epsilon}\mathbf{u}^{\epsilon}||\Delta J_{\epsilon}\mathbf{u}^{\epsilon}|\|_{\mathbb{L}^2}^2+C_{\varepsilon}\|\nabla J_{\epsilon}\mathbf{u}^{\epsilon}\|_{\mathbb{L}^2}^2+C_{\varepsilon}\|\nabla J_{\epsilon}\mathbf{u}^{\epsilon}\|_{\mathbb{L}^2}^{10}.
\end{split}
\end{equation}
By choosing $\varepsilon$ sufficiently small, and using Lemma \ref{lem0} as well as estimate \eqref{lem3-3}, it follows from \eqref{lem3-5} that
\begin{equation*}
\begin{split}
&\|\nabla\mathbf{u}^{\epsilon}(t)\|_{\mathbb{L}^2}^2+\int_0^t\|\nabla\Delta J_{\epsilon}\mathbf{u}^{\epsilon}\|_{\mathbb{L}^2}^2\,\mathrm{d}s\leq\|\nabla\mathbf{u}^{\epsilon}_0\|_{\mathbb{L}^2}^2+C_{\|\mathbf{u}_0\|_{\mathbb{L}^2},T}+C\int_0^t\|\nabla\mathbf{u}^{\epsilon}\|_{\mathbb{L}^2}^{10}\,\mathrm{d}s.
\end{split}
\end{equation*}
Thus, by using the Bihari inequality \cite{bihari1956}, there exists a unique positive constant
\begin{equation*}
\begin{split}
T^*\asymp\frac{1}{\left(\|\nabla\mathbf{u}_0\|_{\mathbb{L}^2}^2+C_{\|\mathbf{u}_0\|_{\mathbb{L}^2},T}\right)^4},
\end{split}
\end{equation*}
such that for any $T'<T^*$ and $t\in[0,T']$,
\begin{equation}\label{lem3-6}
\begin{split}
&\|\nabla\mathbf{u}^{\epsilon}(t)\|_{\mathbb{L}^2}^2+\int_0^t\|\nabla\Delta J_{\epsilon}\mathbf{u}^{\epsilon}\|_{\mathbb{L}^2}^2\,\mathrm{d}s\lesssim_{\|\mathbf{u}_0\|_{\mathbb{H}^1},T'}1.
\end{split}
\end{equation}
Next, multiplying both sides of equation \eqref{Mod-1} by $\Delta^2\mathbf{u}^{\epsilon}$ and using the Sobolev embedding $H^1(\mathbb{R}^3)\hookrightarrow L^6(\mathbb{R}^3)$, the Gagliardo-Nirenberg inequality
\begin{equation}\label{lem3-7}
\begin{split}
&\|f\|_{L^{\infty}}\lesssim\|f\|_{H^1}^{\frac{1}{2}}\|f\|_{H^2}^{\frac{1}{2}},~f\in H^2(\mathbb{R}^3),
\end{split}
\end{equation}
the Sobolev inequality $\|f\|_{L^6}\lesssim\|\nabla f\|_{L^2},~f\in H^1(\mathbb{R}^3)$ \cite{7brezis2011functional} and the fact that $\|D^2 f\|_{L^2}\asymp\|\Delta f\|_{L^2},~f\in H^2(\mathbb{R}^3)$, it follows that
\begin{equation}\label{lem3-8}
\begin{split}
&\frac{1}{2}\frac{\mathrm{d}}{\mathrm{d}t}\|\Delta\mathbf{u}^{\epsilon}\|_{\mathbb{L}^2}^2+\|\Delta^2 J_{\epsilon}\mathbf{u}^{\epsilon}\|_{\mathbb{L}^2}^2=\|\nabla\Delta J_{\epsilon}\mathbf{u}^{\epsilon}\|_{\mathbb{L}^2}^2+2\|\Delta J_{\epsilon}\mathbf{u}^{\epsilon}\|_{\mathbb{L}^2}^2-2(|J_{\epsilon}\mathbf{u}^{\epsilon}|^2J_{\epsilon}\mathbf{u}^{\epsilon},\Delta^2 J_{\epsilon}\mathbf{u}^{\epsilon})_{\mathbb{L}^2}\\
&+2(\Delta(|J_{\epsilon}\mathbf{u}^{\epsilon}|^2J_{\epsilon}\mathbf{u}^{\epsilon}),\Delta^2 J_{\epsilon}\mathbf{u}^{\epsilon})_{\mathbb{L}^2}-(J_{\epsilon}\mathbf{u}^{\epsilon}\times \Delta J_{\epsilon}\mathbf{u}^{\epsilon},\Delta^2 J_{\epsilon}\mathbf{u}^{\epsilon})_{\mathbb{L}^2}\\
&\leq\varepsilon\|\Delta^2 J_{\epsilon}\mathbf{u}^{\epsilon}\|_{\mathbb{L}^2}^2+C_{\varepsilon}\|\Delta \mathbf{u}^{\epsilon}\|_{\mathbb{L}^2}^2+C_{\varepsilon}\|J_{\epsilon}\mathbf{u}^{\epsilon}\|_{\mathbb{L}^6}^6+C_{\varepsilon}\|\Delta(|J_{\epsilon}\mathbf{u}^{\epsilon}|^2J_{\epsilon}\mathbf{u}^{\epsilon})\|_{\mathbb{L}^2}^2\\
&+C_{\varepsilon}\|J_{\epsilon}\mathbf{u}^{\epsilon}\times \Delta J_{\epsilon}\mathbf{u}^{\epsilon}\|_{\mathbb{L}^2}^2\\
&\leq\varepsilon\|\Delta^2 J_{\epsilon}\mathbf{u}^{\epsilon}\|_{\mathbb{L}^2}^2+C_{\varepsilon}\|\Delta \mathbf{u}^{\epsilon}\|_{\mathbb{L}^2}^2+C_{\varepsilon}\|\mathbf{u}^{\epsilon}\|_{\mathbb{H}^1}^6+C_{\varepsilon}\|J_{\epsilon}\mathbf{u}^{\epsilon}\|_{\mathbb{L}^{\infty}}^4\|\Delta\mathbf{u}^{\epsilon}\|_{\mathbb{L}^2}^2\\
&+C_{\varepsilon}\|\nabla J_{\epsilon}\mathbf{u}^{\epsilon}\|_{\mathbb{L}^6}^4\|J_{\epsilon}\mathbf{u}^{\epsilon}\|_{\mathbb{L}^6}^2+C_{\varepsilon}\|J_{\epsilon}\mathbf{u}^{\epsilon}\|_{\mathbb{L}^{\infty}}^2\|\Delta\mathbf{u}^{\epsilon}\|_{\mathbb{L}^{2}}^2\\
&\leq\varepsilon\|\Delta^2 J_{\epsilon}\mathbf{u}^{\epsilon}\|_{\mathbb{L}^2}^2+C_{\varepsilon}\|\mathbf{u}^{\epsilon}\|_{\mathbb{H}^1}^6+C_{\varepsilon}(1+\|J_{\epsilon}\mathbf{u}^{\epsilon}\|_{\mathbb{H}^{1}}^2\|J_{\epsilon}\mathbf{u}^{\epsilon}\|_{\mathbb{H}^{2}}^2+\|J_{\epsilon}\mathbf{u}^{\epsilon}\|_{\mathbb{H}^{2}}^2)\|\Delta\mathbf{u}^{\epsilon}\|_{\mathbb{L}^2}^2.
\end{split}
\end{equation}
Therefore, by choosing sufficiently small $\varepsilon$, applying the Gronwall lemma to \eqref{lem3-8}, and using the estimates \eqref{lem3-3} and \eqref{lem3-6}, we obtain the estimate \eqref{lem3-1}.

Finally, for \eqref{lem3-2}, by applying the estimate \eqref{lem3-1} , we deduce from equation \eqref{Mod-1} that
\begin{equation*}
\begin{split}
&\|\mathbf{u}^{\epsilon}_t\|_{L^2(0,T';\mathbb{L}^{2})}^2\lesssim\|J_{\epsilon}\mathbf{u}^{\epsilon}\|_{L^2(0,T';\mathbb{H}^{4})}^2+\||J_{\epsilon}\mathbf{u}^{\epsilon}|^2J_{\epsilon}\mathbf{u}^{\epsilon}\|_{L^2(0,T';\mathbb{L}^{2})}^2\\
&+\|\Delta(|J_{\epsilon}\mathbf{u}^{\epsilon}|^2J_{\epsilon}\mathbf{u}^{\epsilon})\|_{L^2(0,T';\mathbb{L}^{2})}^2+\|J_{\epsilon}\mathbf{u}^{\epsilon}\times \Delta J_{\epsilon}\mathbf{u}^{\epsilon}\|_{L^2(0,T';\mathbb{L}^{2})}^2\\
&\lesssim_{T'}\|J_{\epsilon}\mathbf{u}^{\epsilon}\|_{L^2(0,T';\mathbb{H}^{4})}^2+\|J_{\epsilon}\mathbf{u}^{\epsilon}\|_{L^{\infty}(0,T';\mathbb{H}^{1})}^6+\|J_{\epsilon}\mathbf{u}^{\epsilon}\|_{L^{\infty}(0,T';\mathbb{H}^{2})}^6+\|J_{\epsilon}\mathbf{u}^{\epsilon}\|_{L^{\infty}(0,T';\mathbb{H}^{2})}^4\\
&\lesssim_{\|\mathbf{u}_0\|_{\mathbb{H}^2},T'}1.
\end{split}
\end{equation*}
The proof of Lemma \ref{lem3} is thus completed.
\end{proof}

In order to ultimately obtain a classical solution, and even any smooth solution, we need to obtain higher-order uniform energy estimates. The detailed result is as follows.
\begin{lemma}\label{lem4} Let $\mathbf{u}_0\in\mathbb{H}^k$ where the integer $k\geq3$. Then there exists $T^*=T^*(\|\nabla\mathbf{u}_0\|_{\mathbb{L}^2})>0$ such that for any $\epsilon\in(0,1)$, $T'\in(0,T^*)$ and $t\in[0,T']$,
\begin{align}
&\|\mathbf{u}^{\epsilon}(t)\|_{\mathbb{H}^k}^{2}+\int_0^t\|J_{\epsilon}\mathbf{u}^{\epsilon}\|_{\mathbb{H}^{k+2}}^2\,\mathrm{d}s\lesssim_{\|\mathbf{u}_0\|_{\mathbb{H}^k},T'}1,\label{lem4-11}\\
&\|\mathbf{u}^{\epsilon}_t\|_{L^2(0,T';\mathbb{H}^{k-2})}^2\lesssim_{\|\mathbf{u}_0\|_{\mathbb{H}^k},T'}1.\label{lem4-22}
\end{align}
\end{lemma}
\begin{proof}[\emph{\textbf{Proof}}] Multiplying both sides of equation \eqref{Mod-1} by $\Lambda^{2k}\mathbf{u}^{\epsilon}$ and using integration by parts, it follows that
\begin{equation*}
\begin{split}
&\frac{1}{2}\frac{\mathrm{d}}{\mathrm{d}t}\|\Lambda^{k}\mathbf{u}^{\epsilon}\|_{\mathbb{L}^2}^2+\|\Delta\Lambda^{k} J_{\epsilon}\mathbf{u}^{\epsilon}\|_{\mathbb{L}^2}^2\\
&=\|\nabla\Lambda^{k} J_{\epsilon}\mathbf{u}^{\epsilon}\|_{\mathbb{L}^2}^2+2\|\Lambda^k J_{\epsilon}\mathbf{u}^{\epsilon}\|_{\mathbb{L}^2}^2-2(\Lambda^{k}(|J_{\epsilon}\mathbf{u}^{\epsilon}|^2J_{\epsilon}\mathbf{u}^{\epsilon}),\Lambda^{k} J_{\epsilon}\mathbf{u}^{\epsilon})_{\mathbb{L}^2}\\
&+2(\Lambda^{k}(|J_{\epsilon}\mathbf{u}^{\epsilon}|^2J_{\epsilon}\mathbf{u}^{\epsilon}),\Delta\Lambda^{k} J_{\epsilon}\mathbf{u}^{\epsilon})_{\mathbb{L}^2}+(\Lambda^{k}(J_{\epsilon}\mathbf{u}^{\epsilon}\times \nabla J_{\epsilon}\mathbf{u}^{\epsilon}),\nabla\Lambda^{k} J_{\epsilon}\mathbf{u}^{\epsilon})_{\mathbb{L}^2}.
\end{split}
\end{equation*}
Since 
\begin{equation}\label{lem4-0}
\begin{split}
&\|\nabla\Lambda^k\cdot\|_{L^2}^2=-(\Lambda^k\cdot,\Delta\Lambda^k\cdot)_{L^2}\leq\varepsilon\|\Delta\Lambda^k\cdot\|_{L^2}^2+C_{\varepsilon}\|\Lambda^k\cdot\|_{L^2}^2,
\end{split}
\end{equation}
it follows that
\begin{equation}\label{lem4-3}
\begin{split}
&\frac{\mathrm{d}}{\mathrm{d}t}\|\mathbf{u}^{\epsilon}\|_{\mathbb{H}^k}^2+\|\Delta J_{\epsilon}\mathbf{u}^{\epsilon}\|_{\mathbb{H}^k}^2\lesssim\|J_{\epsilon}\mathbf{u}^{\epsilon}\|_{\mathbb{H}^k}^2+\||J_{\epsilon}\mathbf{u}^{\epsilon}|^2J_{\epsilon}\mathbf{u}^{\epsilon}\|_{\mathbb{H}^k}\| J_{\epsilon}\mathbf{u}^{\epsilon}\|_{\mathbb{H}^k}\\
&+\||J_{\epsilon}\mathbf{u}^{\epsilon}|^2J_{\epsilon}\mathbf{u}^{\epsilon}\|_{\mathbb{H}^k}^2+|(\Lambda^{k}(J_{\epsilon}\mathbf{u}^{\epsilon}\times \nabla J_{\epsilon}\mathbf{u}^{\epsilon}),\nabla\Lambda^{k} J_{\epsilon}\mathbf{u}^{\epsilon})_{\mathbb{L}^2}|\\
&:=\|J_{\epsilon}\mathbf{u}^{\epsilon}\|_{\mathbb{H}^k}^2+A_1+A_2+A_3.
\end{split}
\end{equation}
For $A_1$, by using the Moser estimate \cite{bahouri2011fourier}
\begin{equation*}
\begin{split}
&\|fg\|_{H^s}\lesssim\|f\|_{L^{\infty}}\|g\|_{H^s}+\|g\|_{L^{\infty}}\|f\|_{H^s},~f,~g\in H^s(\mathbb{R}^d)\cap L^{\infty}(\mathbb{R}^d),~s>0,
\end{split}
\end{equation*}
and inequality \eqref{lem3-7}, it follows that
\begin{equation}\label{lem4-4}
\begin{split}
&A_1\lesssim(\|J_{\epsilon}\mathbf{u}^{\epsilon}\cdot J_{\epsilon}\mathbf{u}^{\epsilon}\|_{\mathbb{H}^k}\|J_{\epsilon}\mathbf{u}^{\epsilon}\|_{\mathbb{L}^{\infty}}+\|J_{\epsilon}\mathbf{u}^{\epsilon}\|_{\mathbb{L}^{\infty}}^2\|J_{\epsilon}\mathbf{u}^{\epsilon}\|_{\mathbb{H}^k})\|J_{\epsilon}\mathbf{u}^{\epsilon}\|_{\mathbb{H}^k}\\
&\lesssim\|J_{\epsilon}\mathbf{u}^{\epsilon}\|_{\mathbb{L}^{\infty}}^2\|J_{\epsilon}\mathbf{u}^{\epsilon}\|_{\mathbb{H}^k}^2\lesssim\|J_{\epsilon}\mathbf{u}^{\epsilon}\|_{\mathbb{H}^{2}}^2\|J_{\epsilon}\mathbf{u}^{\epsilon}\|_{\mathbb{H}^k}^2.
\end{split}
\end{equation}
Similarly, we have
\begin{equation}\label{lem4-5}
\begin{split}
&A_2\lesssim\|J_{\epsilon}\mathbf{u}^{\epsilon}\|_{\mathbb{L}^{\infty}}^4\|J_{\epsilon}\mathbf{u}^{\epsilon}\|_{\mathbb{H}^k}^2\lesssim\|J_{\epsilon}\mathbf{u}^{\epsilon}\|_{\mathbb{H}^{1}}^2\|J_{\epsilon}\mathbf{u}^{\epsilon}\|_{\mathbb{H}^{2}}^2\|J_{\epsilon}\mathbf{u}^{\epsilon}\|_{\mathbb{H}^k}^2.
\end{split}
\end{equation}
For $A_3$, by using the fact that
\begin{equation*}
\begin{split}
&\Lambda^{k}(J_{\epsilon}\mathbf{u}^{\epsilon}\times \nabla J_{\epsilon}\mathbf{u}^{\epsilon})\asymp\Lambda^{k}J_{\epsilon}\mathbf{u}^{\epsilon}\times \nabla J_{\epsilon}\mathbf{u}^{\epsilon}+J_{\epsilon}\mathbf{u}^{\epsilon}\times (\nabla \Lambda^{k}J_{\epsilon}\mathbf{u}^{\epsilon})+C_k\sum_{i=1}^{k-1}\Lambda^{i}J_{\epsilon}\mathbf{u}^{\epsilon}\times(\nabla\Lambda^{k-i}J_{\epsilon}\mathbf{u}^{\epsilon}),
\end{split}
\end{equation*}
the Sobolev embeddings $H^2(\mathbb{R}^3)\hookrightarrow L^{\infty}(\mathbb{R}^3)$ and $H^1(\mathbb{R}^3)\hookrightarrow L^p(\mathbb{R}^3),~p\in[2,6]$, and the estimate \eqref{lem4-0}, it follows that
\begin{equation}\label{lem4-6}
\begin{split}
&A_3\lesssim|(\Lambda^{k}J_{\epsilon}\mathbf{u}^{\epsilon}\times \nabla J_{\epsilon}\mathbf{u}^{\epsilon},\nabla\Lambda^{k} J_{\epsilon}\mathbf{u}^{\epsilon})_{\mathbb{L}^2}|+C_k\sum_{i=1}^{k-1}|(\Lambda^{i}J_{\epsilon}\mathbf{u}^{\epsilon}\times(\nabla\Lambda^{k-i}J_{\epsilon}\mathbf{u}^{\epsilon}),\nabla\Lambda^{k} J_{\epsilon}\mathbf{u}^{\epsilon})_{\mathbb{L}^2}|\\
&\lesssim\|\nabla J_{\epsilon}\mathbf{u}^{\epsilon}\|_{\mathbb{L}^{\infty}}\|\Lambda^{k}J_{\epsilon}\mathbf{u}^{\epsilon}\|_{\mathbb{L}^{2}}\|\nabla\Lambda^{k} J_{\epsilon}\mathbf{u}^{\epsilon}\|_{\mathbb{L}^{2}}\\
&+C_k\sum_{i=1}^{k-1}\|\Lambda^{i}J_{\epsilon}\mathbf{u}^{\epsilon}\|_{\mathbb{L}^3}\|\nabla\Lambda^{k-i}J_{\epsilon}\mathbf{u}^{\epsilon}\|_{\mathbb{L}^6}\|\nabla\Lambda^{k} J_{\epsilon}\mathbf{u}^{\epsilon}\|_{\mathbb{L}^2}\\
&\lesssim\varepsilon\|\nabla\Lambda^{k} J_{\epsilon}\mathbf{u}^{\epsilon}\|_{\mathbb{L}^{2}}^2+C_{\varepsilon}\|J_{\epsilon}\mathbf{u}^{\epsilon}\|_{\mathbb{H}^{3}}^2\|\Lambda^{k}J_{\epsilon}\mathbf{u}^{\epsilon}\|_{\mathbb{L}^{2}}^2+C_{\varepsilon,k}\|\nabla\Lambda^{k} J_{\epsilon}\mathbf{u}^{\epsilon}\|_{\mathbb{L}^{2}}^2\|\Lambda^{k}J_{\epsilon}\mathbf{u}^{\epsilon}\|_{\mathbb{L}^{2}}^2\\
&\lesssim\varepsilon\|\Delta J_{\epsilon}\mathbf{u}^{\epsilon}\|_{\mathbb{H}^k}^2+C_{\varepsilon,k}\| J_{\epsilon}\mathbf{u}^{\epsilon}\|_{\mathbb{H}^{k+1}}^2\|J_{\epsilon}\mathbf{u}^{\epsilon}\|_{\mathbb{H}^{k}}^2.
\end{split}
\end{equation}
Plugging \eqref{lem4-4}, \eqref{lem4-5} and \eqref{lem4-6} into \eqref{lem4-3} and choosing $\varepsilon$ small enough, we infer that
\begin{equation}\label{lem4-7}
\begin{split}
&\frac{\mathrm{d}}{\mathrm{d}t}\|\mathbf{u}^{\epsilon}\|_{\mathbb{H}^k}^2+\|\Delta J_{\epsilon}\mathbf{u}^{\epsilon}\|_{\mathbb{H}^k}^2\\
&\lesssim_k(1+\|J_{\epsilon}\mathbf{u}^{\epsilon}\|_{\mathbb{H}^{2}}^2+\|J_{\epsilon}\mathbf{u}^{\epsilon}\|_{\mathbb{H}^{1}}^2\|J_{\epsilon}\mathbf{u}^{\epsilon}\|_{\mathbb{H}^{2}}^2+\| J_{\epsilon}\mathbf{u}^{\epsilon}\|_{\mathbb{H}^{k+1}}^2)\| J_{\epsilon}\mathbf{u}^{\epsilon}\|_{\mathbb{H}^{k}}^2.
\end{split}
\end{equation}
Note that inequality \eqref{lem4-7} holds for any integer $k\geq3$. First, let $k=3$, and using the Gronwall lemma and the uniform estimate \eqref{lem3-1}, we obtain
\begin{equation*}
\begin{split}
&\|\mathbf{u}^{\epsilon}(t)\|_{\mathbb{H}^3}^{2}+\int_0^t\|J_{\epsilon}\mathbf{u}^{\epsilon}\|_{\mathbb{H}^5}^2\,\mathrm{d}s\lesssim_{\|\mathbf{u}_0\|_{\mathbb{H}^3},T'}1.
\end{split}
\end{equation*}
Furthermore, through a bootstrap argument, we increment $k$ step by step and repeatedly apply the Gronwall lemma to inequality \eqref{lem4-7}, thereby obtaining the estimate \eqref{lem4-11}.

Regarding \eqref{lem4-22}, it follows from \eqref{Mod-1} and Moser estimate that
\begin{equation*}
\begin{split}
&\|\mathbf{u}^{\epsilon}_t\|_{\mathbb{H}^{k-2}}^2\lesssim\|J_{\epsilon}\mathbf{u}\|_{\mathbb{H}^{k+2}}^2+\||J_{\epsilon}\mathbf{u}|^2J_{\epsilon}\mathbf{u}\|_{\mathbb{H}^{k}}^2+\|J_{\epsilon}\mathbf{u}\times\Delta J_{\epsilon}\mathbf{u}\|_{\mathbb{H}^{k-2}}^2\\
&\lesssim_k\|J_{\epsilon}\mathbf{u}\|_{\mathbb{H}^{k+2}}^2+\|J_{\epsilon}\mathbf{u}\|_{\mathbb{H}^{2}}^4\|J_{\epsilon}\mathbf{u}\|_{\mathbb{H}^{k}}^2+\|J_{\epsilon}\mathbf{u}\|_{\mathbb{H}^{k}}^4.
\end{split}
\end{equation*}
This combined with \eqref{lem4-11} means that
\begin{equation*}
\begin{split}
&\|\mathbf{u}^{\epsilon}_t\|_{L^2(0,T';\mathbb{H}^{k-2})}^2\lesssim_k\|J_{\epsilon}\mathbf{u}^{\epsilon}\|_{L^2(0,T';\mathbb{H}^{k+2})}^2+\|J_{\epsilon}\mathbf{u}^{\epsilon}\|_{L^{\infty}(0,T';\mathbb{H}^{2})}^4\|J_{\epsilon}\mathbf{u}^{\epsilon}\|_{L^2(0,T';\mathbb{H}^{k})}^2+\|J_{\epsilon}\mathbf{u}^{\epsilon}\|_{L^{\infty}(0,T';\mathbb{H}^{k})}^4\\
&\lesssim_{\|\mathbf{u}_0\|_{\mathbb{H}^k},T'}1.
\end{split}
\end{equation*}
The proof of lemma \ref{lem4} is thus completed.
\end{proof}

\section{Local existence of solutions}\label{sec4}
This section is devoted to obtaining local strong (smooth) solutions to the original system by taking the limit as $\epsilon \searrow 0$, utilizing the uniform estimates established in Section~\ref{sec3}. Moreover, we derive the corresponding blow-up criterion, which lays the groundwork for the subsequent proof of the global existence of solutions.
\subsection{Local strong solutions}
By applying the uniform boundedness estimates \eqref{lem3-1} and \eqref{lem3-2}, together with the Banach-Alaoglu theorem and the uniqueness of the limit, it follows that
\begin{equation}\label{4-1}
\begin{split}
&\mathbf{u}^{\epsilon}\rightarrow\mathbf{u}~\textrm{weakly-star in}~L^{\infty}(0,T';\mathbb{H}^2),\\
&\frac{\mathrm{d}}{\mathrm{d}t}J_{\epsilon}\mathbf{u}^{\epsilon}\rightarrow \frac{\mathrm{d}}{\mathrm{d}t}\mathbf{u}~\textrm{weakly in}~L^{2}(0,T';\mathbb{L}^2),\\
&J_{\epsilon}\mathbf{u}^{\epsilon}\rightarrow\mathbf{u}~\textrm{weakly in}~L^{2}(0,T';\mathbb{H}^4).
\end{split}
\end{equation}
Since $\mathbb{H}^4\hookrightarrow\hookrightarrow\mathbb{H}^3_{loc}\hookrightarrow\mathbb{H}^2_{loc}$, by using the Aubin-Lions lemma we see from \eqref{4-1} that
\begin{equation}\label{4-2}
\begin{split}
&J_{\epsilon}\mathbf{u}^{\epsilon}\rightarrow\mathbf{u}~\textrm{in}~L^{2}_w(0,T';\mathbb{H}^4)\cap L^{2}(0,T';\mathbb{H}^3_{loc}).
\end{split}
\end{equation}
Moreover, invoking the Fatou lemma, it follows that
\begin{equation}\label{4-3}
\begin{split}
&\|\mathbf{u}(t)\|_{\mathbb{H}^2}^{2}+\int_0^t\|\mathbf{u}\|_{\mathbb{H}^4}^2\,\mathrm{d}s\lesssim_{\|\mathbf{u}_0\|_{\mathbb{H}^2},T'}1,~\|\mathbf{u}_t\|_{L^2(0,T';\mathbb{L}^{2})}^2\lesssim_{\|\mathbf{u}_0\|_{\mathbb{H}^2},T'}1.
\end{split}
\end{equation}
Let
\begin{equation*}
\begin{split}
&F(\mathbf{u}):=-\Delta^2\mathbf{u}-\Delta\mathbf{u}+2(1-|\mathbf{u}|^2)\mathbf{u}+2\Delta(|\mathbf{u}|^2\mathbf{u})-\mathbf{u}\times\Delta \mathbf{u}.
\end{split}
\end{equation*}
We have the following convergence result.
\begin{proposition}\label{pro1} Under the same assumptions as Lemma \ref{lem3}, for any $t\in[0,T']$ and $\phi\in \mathbb{L}^2$, it follows that
\begin{equation}\label{4-4}
\begin{split}
&0=\lim_{\epsilon\searrow0}(\mathbf{u}^{\epsilon}(t),\phi)_{\mathbb{L}^2}-(\mathbf{u}^{\epsilon}_0,\phi)_{\mathbb{L}^2}-\int_0^t(F^{\epsilon}(\mathbf{u}^{\epsilon}(s)),\phi)_{\mathbb{L}^2}\,\mathrm{d}s\\
&=(\mathbf{u}(t),\phi)_{\mathbb{L}^2}-(\mathbf{u}_0,\phi)_{\mathbb{L}^2}-\int_0^t(F(\mathbf{u}(s)),\phi)_{\mathbb{L}^2}\,\mathrm{d}s,
\end{split}
\end{equation}
where $F^{\epsilon}(\mathbf{u}^{\epsilon})$ is defined by \eqref{def1}. Moreover, the solution $\mathbf{u}$ satisfies the following regularity
\begin{equation*}
\begin{split}
&\mathbf{u}\in \mathcal{C}([0,T'];\mathbb{H}^2)\cap L^2(0,T';\mathbb{H}^4).
\end{split}
\end{equation*}
\end{proposition}
\begin{proof}[\emph{\textbf{Proof}}] By using the fact that $\mathbf{u}^{\epsilon}\rightarrow\mathbf{u}~\textrm{weakly-star in}~L^{\infty}(0,T';\mathbb{H}^2)$ and Lemma \ref{lem0}, it is obvious that
\begin{equation*}
\begin{split}
&\lim_{\epsilon\searrow0}(\mathbf{u}^{\epsilon}(t),\phi)_{\mathbb{L}^2}=(\mathbf{u}(t),\phi)_{\mathbb{L}^2}~\textrm{and}~\lim_{\epsilon\searrow0}(\mathbf{u}^{\epsilon}_0,\phi)_{\mathbb{L}^2}=(\mathbf{u}_0,\phi)_{\mathbb{L}^2}.
\end{split}
\end{equation*}
Since $J_{\epsilon}\mathbf{u}^{\epsilon}\rightarrow\mathbf{u}~\textrm{weakly in}~L^{2}(0,T';\mathbb{H}^4)$, it follows from Lemma \ref{lem0} and \eqref{lem3-1} that
\begin{equation*}
\begin{split}
&\left|\int_0^t(J_{\epsilon}\Delta^2J_{\epsilon}\mathbf{u}^{\epsilon}(s),\phi)_{\mathbb{L}^2}-(\Delta^2\mathbf{u}(s),\phi)_{\mathbb{L}^2}\,\mathrm{d}s\right|\\
&\leq\left|\int_0^t(\Delta^2J_{\epsilon}\mathbf{u}^{\epsilon}(s),J_{\epsilon}\phi-\phi)_{\mathbb{L}^2}\,\mathrm{d}s\right|+\left|\int_0^t(\Delta^2J_{\epsilon}\mathbf{u}^{\epsilon}(s),\phi)_{\mathbb{L}^2}-(\Delta^2\mathbf{u}(s),\phi)_{\mathbb{L}^2}\,\mathrm{d}s\right|\\
&\leq C_{T'}\|J_{\epsilon}\mathbf{u}^{\epsilon}\|_{L^2(0,T';\mathbb{H}^4)}\|J_{\epsilon}\phi-\phi\|_{\mathbb{L}^2}+\left|\int_0^t(\Delta^2J_{\epsilon}\mathbf{u}^{\epsilon}(s),\phi)_{\mathbb{L}^2}-(\Delta^2\mathbf{u}(s),\phi)_{\mathbb{L}^2}\,\mathrm{d}s\right|\\
&\rightarrow0~\textrm{as}~\epsilon\searrow0.
\end{split}
\end{equation*}
For the remaining terms, we shall focus on proving the following two convergence results
\begin{align}
&\lim_{\epsilon\searrow0}\int_0^t(J_{\epsilon}\Delta(|J_{\epsilon}\mathbf{u}^{\epsilon}|^2J_{\epsilon}\mathbf{u}^{\epsilon}),\phi)_{\mathbb{L}^2}\,\mathrm{d}s=\int_0^t(\Delta(|\mathbf{u}|^2\mathbf{u}),\phi)_{\mathbb{L}^2}\,\mathrm{d}s,\label{4-5}\\
&\lim_{\epsilon\searrow0}\int_0^t(J_{\epsilon}(J_{\epsilon}\mathbf{u}^{\epsilon}\times \Delta J_{\epsilon}\mathbf{u}^{\epsilon}),\phi)_{\mathbb{L}^2}\,\mathrm{d}s=\int_0^t(\mathbf{u}\times\Delta \mathbf{u},\phi)_{\mathbb{L}^2}\,\mathrm{d}s,\label{4-6}
\end{align}
and the proofs for the other terms are similar and therefore omitted. For every $\phi\in L^2(\mathbb{R}^3)$, there exists $\phi_{h}\in \mathcal{C}_0^{\infty}(\mathbb{R}^3)$ such that $\|\phi-\phi_h\|_{L^2}\leq h$ for $h>0$. Moreover, for any $\phi_h\in \mathcal{C}_0^{\infty}(\mathbb{R}^3)$, there exists $d>0$ such that supp$~\phi_h$ is a compact subset of $\mathcal{O}_d$. Thus by using the triangle inequality and  noting that $H^2(\mathbb{R}^3)$ is a Banach algebra, it follows that
\begin{equation}\label{4-7}
\begin{split}
&\left|(J_{\epsilon}\Delta(|J_{\epsilon}\mathbf{u}^{\epsilon}|^2J_{\epsilon}\mathbf{u}^{\epsilon}),\phi)_{\mathbb{L}^2}-(\Delta(|\mathbf{u}|^2\mathbf{u}),\phi)_{\mathbb{L}^2}\right|\\
&\leq\left|(\Delta(|J_{\epsilon}\mathbf{u}^{\epsilon}|^2J_{\epsilon}\mathbf{u}^{\epsilon}),J_{\epsilon}\phi-\phi)_{\mathbb{L}^2}\right|+\left|(\Delta(|J_{\epsilon}\mathbf{u}^{\epsilon}|^2J_{\epsilon}\mathbf{u}^{\epsilon}-|\mathbf{u}|^2\mathbf{u}),\phi-\phi_{h})_{\mathbb{L}^2}\right|\\
&+\left|(\Delta(|J_{\epsilon}\mathbf{u}^{\epsilon}|^2J_{\epsilon}\mathbf{u}^{\epsilon}-|\mathbf{u}|^2\mathbf{u}),\phi_{h})_{\mathbb{L}^2}\right|\\
&\lesssim\|J_{\epsilon}\mathbf{u}^{\epsilon}\|_{\mathbb{H}^2}^3\|J_{\epsilon}\phi-\phi\|_{\mathbb{L}^2}+(\|J_{\epsilon}\mathbf{u}^{\epsilon}\|_{\mathbb{H}^2}^3+\|\mathbf{u}\|_{\mathbb{H}^2}^3)\|\phi-\phi_h\|_{\mathbb{L}^2}\\
&+(\|J_{\epsilon}\mathbf{u}^{\epsilon}\|_{\mathbb{H}^2(\mathcal{O}_d)}^2+\|\mathbf{u}\|_{\mathbb{H}^2(\mathcal{O}_d)}^2)\|J_{\epsilon}\mathbf{u}^{\epsilon}-\mathbf{u}\|_{\mathbb{H}^2(\mathcal{O}_d)}\|\phi_h\|_{\mathbb{L}^2(\mathcal{O}_d)}.
\end{split}
\end{equation}
Thus by using the estimates \eqref{lem3-1} and \eqref{4-3} and applying the convergence result \eqref{4-2}, we see from \eqref{4-7} that
\begin{equation*}
\begin{split}
\limsup_{\epsilon\searrow0}\left|\int_0^t(J_{\epsilon}\Delta(|J_{\epsilon}\mathbf{u}^{\epsilon}|^2J_{\epsilon}\mathbf{u}^{\epsilon}),\phi)_{\mathbb{L}^2}-(\Delta(|\mathbf{u}|^2\mathbf{u}),\phi)_{\mathbb{L}^2}\,\mathrm{d}s\right|\leq C_{\|\mathbf{u}_0\|_{\mathbb{H}^2},T'}h,
\end{split}
\end{equation*}
which combined with the arbitrariness of $h$ implies that \eqref{4-5} holds. Similarly, it follows that
\begin{equation*}
\begin{split}
&\left|(J_{\epsilon}(J_{\epsilon}\mathbf{u}^{\epsilon}\times \Delta J_{\epsilon}\mathbf{u}^{\epsilon}),\phi)_{\mathbb{L}^2}-(\mathbf{u}\times\Delta \mathbf{u},\phi)_{\mathbb{L}^2}\right|\\
&\leq\left|(J_{\epsilon}\mathbf{u}^{\epsilon}\times \Delta J_{\epsilon}\mathbf{u}^{\epsilon},J_{\epsilon}\phi-\phi)_{\mathbb{L}^2}\right|+\left|(J_{\epsilon}\mathbf{u}^{\epsilon}\times \Delta J_{\epsilon}\mathbf{u}^{\epsilon}-\mathbf{u}\times\Delta \mathbf{u},\phi-\phi_{h})_{\mathbb{L}^2}\right|\\
&+\left|(J_{\epsilon}\mathbf{u}^{\epsilon}\times \Delta J_{\epsilon}\mathbf{u}^{\epsilon}-\mathbf{u}\times\Delta \mathbf{u},\phi_{h})_{\mathbb{L}^2}\right|\\
&\lesssim\|J_{\epsilon}\mathbf{u}^{\epsilon}\|_{\mathbb{H}^2}^2\|J_{\epsilon}\phi-\phi\|_{\mathbb{L}^2}+(\|J_{\epsilon}\mathbf{u}^{\epsilon}\|_{\mathbb{H}^2}^2+\|\mathbf{u}\|_{\mathbb{H}^2}^2)\|\phi-\phi_h\|_{\mathbb{L}^2}\\
&+(\|J_{\epsilon}\mathbf{u}^{\epsilon}\|_{\mathbb{H}^2(\mathcal{O}_d)}+\|\mathbf{u}\|_{\mathbb{H}^2(\mathcal{O}_d)})\|J_{\epsilon}\mathbf{u}^{\epsilon}-\mathbf{u}\|_{\mathbb{H}^2(\mathcal{O}_d)}\|\phi_h\|_{\mathbb{L}^2(\mathcal{O}_d)}.
\end{split}
\end{equation*}
By taking the limit, we have
\begin{equation*}
\begin{split}
\limsup_{\epsilon\searrow0}\left|\int_0^t(J_{\epsilon}(J_{\epsilon}\mathbf{u}^{\epsilon}\times \Delta J_{\epsilon}\mathbf{u}^{\epsilon}),\phi)_{\mathbb{L}^2}-(\mathbf{u}\times\Delta \mathbf{u},\phi)_{\mathbb{L}^2}\,\mathrm{d}s\right|\leq C_{\|\mathbf{u}_0\|_{\mathbb{H}^2},T'}h,
\end{split}
\end{equation*}
which implies that \eqref{4-6} holds. Combining the above convergence results, we finally obtain that the solution $\mathbf{u}$ satisfies \eqref{4-4}. Moreover, according to \eqref{4-1}, we see that $\mathbf{u}\in L^{2}(0,T';\mathbb{H}^4)$ and $\mathbf{u}_t\in L^{2}(0,T';\mathbb{L}^2)$. Thus by using Lions-Magenes lemma \cite{lions1963problemes}, it follows that $\mathbf{u}\in\mathcal{C}([0,T'];[\mathbb{L}^2,\mathbb{H}^4]_{\frac{1}{2}})=\mathcal{C}([0,T'];\mathbb{H}^2)$. The proof of Proposition \ref{pro1} is thus completed.
\end{proof}

Based on the construction in Lemma \ref{lem3-1}, we know that the existence time $T'$ of the solution $\mathbf{u}$ is controlled by
\begin{equation*}
\begin{split}
&T'<\frac{C}{\left(\|\nabla\mathbf{u}_0\|_{\mathbb{L}^2}^2+C_{\|\mathbf{u}_0\|_{\mathbb{L}^2},T}\right)^4}.
\end{split}
\end{equation*}
At time $T'$, choose $\mathbf{u}(T',x)$ as initial data for a new solution and repeat the process by continuing the solution on a time interval $[T',T'']$ for which
\begin{equation*}
\begin{split}
&T''<\frac{C}{\left(\|\nabla\mathbf{u}(T')\|_{\mathbb{L}^2}^2+C_{\|\mathbf{u}(T')\|_{\mathbb{L}^2},T}\right)^4}.
\end{split}
\end{equation*}
Obviously the process can be continued either for any time $T>0$ or until $\|\nabla\mathbf{u}(t,\cdot)\|_{\mathbb{L}^2}$ becomes infinite. Therefore, we have the following blow-up criterion.
\begin{corollary}\label{cor1}
Under the same assumptions as Lemma \ref{lem3}, there exists a maximal time of existence $T_{M}>0$ and a strong solution $\mathbf{u}\in \mathcal{C}([0,T_{M});\mathbb{H}^2)\cap L^2(0,T_M;\mathbb{H}^4)$ to the system \eqref{sys1}. Moreover, for any given $T>0$, if
$
\sup_{t\in[0,T]}\|\nabla \mathbf{u}(t)\|_{\mathbb{L}^2} < \infty,
$
then necessarily $T_M \geq T$ (the solution can be extended to $[0,T]$). Equivalently, if $T_M\leq T$ then necessarily
$
\lim_{t\nearrow T_M}\|\nabla \mathbf{u}(t)\|_{\mathbb{L}^2}=\infty.
$
\end{corollary}

\subsection{Local smooth solutions}
Based on the uniform bound provided by Lemma \ref{lem3-2}, we shall prove that the sequence $\mathbf{u}^{\epsilon}$ is a Cauchy sequence in $\mathcal{C}([0,T'];\mathbb{L}^2)$ when $\mathbf{u}_0\in\mathbb{H}^k,~k\geq6$. Specifically, we have the following result.
\begin{lemma}\label{lem4-1} Let $\mathbf{u}_0\in\mathbb{H}^k$ where the integer $k\geq6$. The family $\mathbf{u}^{\epsilon}$ forms a Cauchy sequence in $\mathcal{C}([0,T'];\mathbb{L}^2)$. In particular, for all $\epsilon$ and $\epsilon'>0$, we have
\begin{equation*}
\begin{split}
&\sup_{t\in[0,T']}\|\mathbf{u}^{\epsilon}-\mathbf{u}^{\epsilon'}\|_{\mathbb{L}^2}\lesssim_{\|\mathbf{u}_0\|_{\mathbb{H}^k},T'}\max\{\epsilon,\epsilon'\}.
\end{split}
\end{equation*}
\end{lemma}
\begin{proof}[\emph{\textbf{Proof}}] By equation \eqref{Mod-1}, it follows that
\begin{equation}\label{4-8}
\begin{split}
&\frac{1}{2}\frac{\mathrm{d}}{\mathrm{d}t}\|\mathbf{u}^{\epsilon}-\mathbf{u}^{\epsilon'}\|_{\mathbb{L}^2}^2=-(J_{\epsilon}^2\Delta^2\mathbf{u}^{\epsilon}-J_{\epsilon'}^2\Delta^2\mathbf{u}^{\epsilon'},\mathbf{u}^{\epsilon}-\mathbf{u}^{\epsilon'})_{\mathbb{L}^2}\\
&-(J_{\epsilon}^2\Delta\mathbf{u}^{\epsilon}-J_{\epsilon'}^2\Delta\mathbf{u}^{\epsilon'},\mathbf{u}^{\epsilon}-\mathbf{u}^{\epsilon'})_{\mathbb{L}^2}+2(J_{\epsilon}^2\mathbf{u}^{\epsilon}-J_{\epsilon'}^2\mathbf{u}^{\epsilon'},\mathbf{u}^{\epsilon}-\mathbf{u}^{\epsilon'})_{\mathbb{L}^2}\\
&-2(J_{\epsilon}(|J_{\epsilon}\mathbf{u}^{\epsilon}|^2J_{\epsilon}\mathbf{u}^{\epsilon})-J_{\epsilon'}(|J_{\epsilon'}\mathbf{u}^{\epsilon'}|^2J_{\epsilon'}\mathbf{u}^{\epsilon'}),\mathbf{u}^{\epsilon}-\mathbf{u}^{\epsilon'})_{\mathbb{L}^2}\\
&+2(J_{\epsilon}\Delta(|J_{\epsilon}\mathbf{u}^{\epsilon}|^2J_{\epsilon}\mathbf{u}^{\epsilon})-J_{\epsilon'}\Delta(|J_{\epsilon'}\mathbf{u}^{\epsilon'}|^2J_{\epsilon'}\mathbf{u}^{\epsilon'}),\mathbf{u}^{\epsilon}-\mathbf{u}^{\epsilon'})_{\mathbb{L}^2}\\
&-(J_{\epsilon}(J_{\epsilon}\mathbf{u}^{\epsilon}\times \Delta J_{\epsilon}\mathbf{u}^{\epsilon})-J_{\epsilon'}(J_{\epsilon'}\mathbf{u}^{\epsilon'}\times \Delta J_{\epsilon'}\mathbf{u}^{\epsilon'}),\mathbf{u}^{\epsilon}-\mathbf{u}^{\epsilon'})_{\mathbb{L}^2}\\
&:=K_1+\cdots+K_6.
\end{split}
\end{equation}
By integrating by parts and using Lemma \ref{lem0}, it follows that
\begin{equation*}
\begin{split}
&K_1=-((J_{\epsilon}^2-J_{\epsilon'}^2)\Delta^2\mathbf{u}^{\epsilon},\mathbf{u}^{\epsilon}-\mathbf{u}^{\epsilon'})_{\mathbb{L}^2}-\|J_{\epsilon'}\Delta(\mathbf{u}^{\epsilon}-\mathbf{u}^{\epsilon'})\|_{\mathbb{L}^2}^2\\
&\leq\|(J_{\epsilon}^2-J_{\epsilon'}^2)\Delta^2\mathbf{u}^{\epsilon}\|_{\mathbb{L}^2}\|\mathbf{u}^{\epsilon}-\mathbf{u}^{\epsilon'}\|_{\mathbb{L}^2}-\|J_{\epsilon'}\Delta(\mathbf{u}^{\epsilon}-\mathbf{u}^{\epsilon'})\|_{\mathbb{L}^2}^2\\
&\lesssim(\|J_{\epsilon}\Delta^2\mathbf{u}^{\epsilon}-\Delta^2\mathbf{u}^{\epsilon}\|_{\mathbb{L}^2}+\|\Delta^2\mathbf{u}^{\epsilon}-J_{\epsilon'}\Delta^2\mathbf{u}^{\epsilon}\|_{\mathbb{L}^2})\|\mathbf{u}^{\epsilon}-\mathbf{u}^{\epsilon'}\|_{\mathbb{L}^2}-\|J_{\epsilon'}\Delta(\mathbf{u}^{\epsilon}-\mathbf{u}^{\epsilon'})\|_{\mathbb{L}^2}^2\\
&\lesssim\max\{\epsilon,\epsilon'\}\|\mathbf{u}^{\epsilon}\|_{\mathbb{H}^6}\|\mathbf{u}^{\epsilon}-\mathbf{u}^{\epsilon'}\|_{\mathbb{L}^2}-\|J_{\epsilon'}\Delta(\mathbf{u}^{\epsilon}-\mathbf{u}^{\epsilon'})\|_{\mathbb{L}^2}^2.
\end{split}
\end{equation*}
Similarly, by using Young's inequality, it follows that
\begin{equation*}
\begin{split}
&K_2+K_3=-((J_{\epsilon}^2-J_{\epsilon'}^2)\Delta\mathbf{u}^{\epsilon},\mathbf{u}^{\epsilon}-\mathbf{u}^{\epsilon'})_{\mathbb{L}^2}-(J_{\epsilon'}^2\Delta(\mathbf{u}^{\epsilon}-\mathbf{u}^{\epsilon'}),\mathbf{u}^{\epsilon}-\mathbf{u}^{\epsilon'})_{\mathbb{L}^2}\\
&+((J_{\epsilon}^2-J_{\epsilon'}^2)\mathbf{u}^{\epsilon},\mathbf{u}^{\epsilon}-\mathbf{u}^{\epsilon'})_{\mathbb{L}^2}+\|J_{\epsilon'}(\mathbf{u}^{\epsilon}-\mathbf{u}^{\epsilon'})\|_{\mathbb{L}^2}^2\\
&\lesssim\max\{\epsilon,\epsilon'\}\|\mathbf{u}^{\epsilon}\|_{\mathbb{H}^3}\|\mathbf{u}^{\epsilon}-\mathbf{u}^{\epsilon'}\|_{\mathbb{L}^2}+C_{\varepsilon}\|\mathbf{u}^{\epsilon}-\mathbf{u}^{\epsilon'}\|_{\mathbb{L}^2}^2+\varepsilon\|J_{\epsilon'}\Delta(\mathbf{u}^{\epsilon}-\mathbf{u}^{\epsilon'})\|_{\mathbb{L}^2}^2.
\end{split}
\end{equation*}
For $K_4$, noting $H^2(\mathbb{R}^3)$ is a Banach algebra and $H^2(\mathbb{R}^3)\hookrightarrow L^{\infty}(\mathbb{R}^3)$, we have
\begin{equation*}
\begin{split}
&K_4=-2((J_{\epsilon}-J_{\epsilon'})(|J_{\epsilon}\mathbf{u}^{\epsilon}|^2J_{\epsilon}\mathbf{u}^{\epsilon}),\mathbf{u}^{\epsilon}-\mathbf{u}^{\epsilon'})_{\mathbb{L}^2}\\
&-2(J_{\epsilon'}(|J_{\epsilon}\mathbf{u}^{\epsilon}|^2(J_{\epsilon}-J_{\epsilon'})\mathbf{u}^{\epsilon}),\mathbf{u}^{\epsilon}-\mathbf{u}^{\epsilon'})_{\mathbb{L}^2}\\
&-2(J_{\epsilon'}(|J_{\epsilon}\mathbf{u}^{\epsilon}|(|J_{\epsilon}\mathbf{u}^{\epsilon}|-|J_{\epsilon'}\mathbf{u}^{\epsilon}|))J_{\epsilon'}\mathbf{u}^{\epsilon'}),\mathbf{u}^{\epsilon}-\mathbf{u}^{\epsilon'})_{\mathbb{L}^2}\\
&-2(J_{\epsilon'}(|J_{\epsilon}\mathbf{u}^{\epsilon}|(|J_{\epsilon'}\mathbf{u}^{\epsilon}|-|J_{\epsilon'}\mathbf{u}^{\epsilon'}|))J_{\epsilon'}\mathbf{u}^{\epsilon'}),\mathbf{u}^{\epsilon}-\mathbf{u}^{\epsilon'})_{\mathbb{L}^2}\\
&-2(J_{\epsilon'}(|J_{\epsilon'}\mathbf{u}^{\epsilon'}|(|J_{\epsilon}\mathbf{u}^{\epsilon}|-|J_{\epsilon'}\mathbf{u}^{\epsilon}|))J_{\epsilon'}\mathbf{u}^{\epsilon'}),\mathbf{u}^{\epsilon}-\mathbf{u}^{\epsilon'})_{\mathbb{L}^2}\\
&-2(J_{\epsilon'}(|J_{\epsilon'}\mathbf{u}^{\epsilon'}|(|J_{\epsilon'}\mathbf{u}^{\epsilon}|-|J_{\epsilon'}\mathbf{u}^{\epsilon'}|))J_{\epsilon'}\mathbf{u}^{\epsilon'}),\mathbf{u}^{\epsilon}-\mathbf{u}^{\epsilon'})_{\mathbb{L}^2}\\
&\lesssim\max\{\epsilon,\epsilon'\}(\|\mathbf{u}^{\epsilon}\|_{\mathbb{H}^2}^3+\|\mathbf{u}^{\epsilon}\|_{\mathbb{H}^2}^2\|\mathbf{u}^{\epsilon'}\|_{\mathbb{H}^2}+\|\mathbf{u}^{\epsilon'}\|_{\mathbb{H}^2}^2\|\mathbf{u}^{\epsilon}\|_{\mathbb{H}^2})\|\mathbf{u}^{\epsilon}-\mathbf{u}^{\epsilon'}\|_{\mathbb{L}^2}\\
&+(\|\mathbf{u}^{\epsilon}\|_{\mathbb{H}^2}^2+\|\mathbf{u}^{\epsilon'}\|_{\mathbb{H}^2}^2)\|\mathbf{u}^{\epsilon}-\mathbf{u}^{\epsilon'}\|_{\mathbb{L}^2}^2.
\end{split}
\end{equation*}
For $K_5$, by integrating by parts, applying the basic equality $|f|^2-|g|^2=f\cdot(f-g)+(f-g)\cdot g$  and using Young's inequality, it follows that
\begin{equation*}
\begin{split}
&K_5=2((J_{\epsilon}-J_{\epsilon'})\Delta(|J_{\epsilon}\mathbf{u}^{\epsilon}|^2J_{\epsilon}\mathbf{u}^{\epsilon}),\mathbf{u}^{\epsilon}-\mathbf{u}^{\epsilon'})_{\mathbb{L}^2}\\
&+2(J_{\epsilon'}(|J_{\epsilon}\mathbf{u}^{\epsilon}|^2J_{\epsilon}\mathbf{u}^{\epsilon}-|J_{\epsilon'}\mathbf{u}^{\epsilon'}|^2J_{\epsilon'}\mathbf{u}^{\epsilon'}),\Delta(\mathbf{u}^{\epsilon}-\mathbf{u}^{\epsilon'}))_{\mathbb{L}^2}\\
&=2((J_{\epsilon}-J_{\epsilon'})\Delta(|J_{\epsilon}\mathbf{u}^{\epsilon}|^2J_{\epsilon}\mathbf{u}^{\epsilon}),\mathbf{u}^{\epsilon}-\mathbf{u}^{\epsilon'})_{\mathbb{L}^2}\\
&+2(J_{\epsilon'}\Delta(|J_{\epsilon}\mathbf{u}^{\epsilon}|^2(J_{\epsilon}-J_{\epsilon'})\mathbf{u}^{\epsilon}),\mathbf{u}^{\epsilon}-\mathbf{u}^{\epsilon'})_{\mathbb{L}^2}\\
&+2(J_{\epsilon'}\Delta((J_{\epsilon}\mathbf{u}^{\epsilon}\cdot(J_{\epsilon}\mathbf{u}^{\epsilon}-J_{\epsilon'}\mathbf{u}^{\epsilon}))J_{\epsilon'}\mathbf{u}^{\epsilon'}),\mathbf{u}^{\epsilon}-\mathbf{u}^{\epsilon'})_{\mathbb{L}^2}\\
&+2((J_{\epsilon}\mathbf{u}^{\epsilon}\cdot(J_{\epsilon'}\mathbf{u}^{\epsilon}-J_{\epsilon'}\mathbf{u}^{\epsilon'}))J_{\epsilon'}\mathbf{u}^{\epsilon'},J_{\epsilon'}\Delta(\mathbf{u}^{\epsilon}-\mathbf{u}^{\epsilon'}))_{\mathbb{L}^2}\\
&+2(J_{\epsilon'}\Delta((J_{\epsilon'}\mathbf{u}^{\epsilon'}\cdot(J_{\epsilon}\mathbf{u}^{\epsilon}-J_{\epsilon'}\mathbf{u}^{\epsilon}))J_{\epsilon'}\mathbf{u}^{\epsilon'}),\mathbf{u}^{\epsilon}-\mathbf{u}^{\epsilon'})_{\mathbb{L}^2}\\
&+2((J_{\epsilon'}\mathbf{u}^{\epsilon'}\cdot(J_{\epsilon'}\mathbf{u}^{\epsilon}-J_{\epsilon'}\mathbf{u}^{\epsilon'}))J_{\epsilon'}\mathbf{u}^{\epsilon'},J_{\epsilon'}\Delta(\mathbf{u}^{\epsilon}-\mathbf{u}^{\epsilon'}))_{\mathbb{L}^2}\\
&\lesssim\max\{\epsilon,\epsilon'\}\||J_{\epsilon}\mathbf{u}^{\epsilon}|^2J_{\epsilon}\mathbf{u}^{\epsilon}\|_{\mathbb{H}^3}\|\mathbf{u}^{\epsilon}-\mathbf{u}^{\epsilon'}\|_{\mathbb{L}^2}+\||J_{\epsilon}\mathbf{u}^{\epsilon}|^2(J_{\epsilon}-J_{\epsilon'})\mathbf{u}^{\epsilon}\|_{\mathbb{H}^2}\|\mathbf{u}^{\epsilon}-\mathbf{u}^{\epsilon'}\|_{\mathbb{L}^2}\\
&+\|(J_{\epsilon}\mathbf{u}^{\epsilon}\cdot(J_{\epsilon}\mathbf{u}^{\epsilon}-J_{\epsilon'}\mathbf{u}^{\epsilon}))J_{\epsilon'}\mathbf{u}^{\epsilon'}\|_{\mathbb{H}^2}\|\mathbf{u}^{\epsilon}-\mathbf{u}^{\epsilon'}\|_{\mathbb{L}^2}+\varepsilon\|J_{\epsilon'}\Delta(\mathbf{u}^{\epsilon}-\mathbf{u}^{\epsilon'})\|_{\mathbb{L}^2}^2\\
&+C_{\varepsilon}\|J_{\epsilon}\mathbf{u}^{\epsilon}\|_{\mathbb{L}^{\infty}}^2\|J_{\epsilon'}\mathbf{u}^{\epsilon'}\|_{\mathbb{L}^{\infty}}^2\|\mathbf{u}^{\epsilon}-\mathbf{u}^{\epsilon'}\|_{\mathbb{L}^2}^2+\|(J_{\epsilon'}\mathbf{u}^{\epsilon'}\cdot(J_{\epsilon}\mathbf{u}^{\epsilon}-J_{\epsilon'}\mathbf{u}^{\epsilon}))J_{\epsilon'}\mathbf{u}^{\epsilon'}\|_{\mathbb{H}^2}\|\mathbf{u}^{\epsilon}-\mathbf{u}^{\epsilon'}\|_{\mathbb{L}^2}\\
&+C_{\varepsilon}\|J_{\epsilon'}\mathbf{u}^{\epsilon'}\|_{\mathbb{L}^{\infty}}^4\|\mathbf{u}^{\epsilon}-\mathbf{u}^{\epsilon'}\|_{\mathbb{L}^2}^2\\
&\lesssim\max\{\epsilon,\epsilon'\}(\|\mathbf{u}^{\epsilon}\|_{\mathbb{H}^3}^3+\|\mathbf{u}^{\epsilon}\|_{\mathbb{H}^3}\|\mathbf{u}^{\epsilon}\|_{\mathbb{H}^2}\|\mathbf{u}^{\epsilon'}\|_{\mathbb{H}^2}+\|\mathbf{u}^{\epsilon}\|_{\mathbb{H}^3}\|\mathbf{u}^{\epsilon'}\|_{\mathbb{H}^2}^2)\|\mathbf{u}^{\epsilon}-\mathbf{u}^{\epsilon'}\|_{\mathbb{L}^2}\\
&+C_{\varepsilon}(\|\mathbf{u}^{\epsilon}\|_{\mathbb{H}^{2}}^4+\|\mathbf{u}^{\epsilon'}\|_{\mathbb{H}^{2}}^4)\|\mathbf{u}^{\epsilon}-\mathbf{u}^{\epsilon'}\|_{\mathbb{L}^2}^2+\varepsilon\|J_{\epsilon'}\Delta(\mathbf{u}^{\epsilon}-\mathbf{u}^{\epsilon'})\|_{\mathbb{L}^2}^2.
\end{split}
\end{equation*}
For $K_6$, it follows that
\begin{equation*}
\begin{split}
&K_6=-((J_{\epsilon}-J_{\epsilon'})(J_{\epsilon}\mathbf{u}^{\epsilon}\times \Delta J_{\epsilon}\mathbf{u}^{\epsilon})),\mathbf{u}^{\epsilon}-\mathbf{u}^{\epsilon'})_{\mathbb{L}^2}\\
&-(J_{\epsilon'}(J_{\epsilon}\mathbf{u}^{\epsilon}\times \Delta J_{\epsilon}\mathbf{u}^{\epsilon}-J_{\epsilon'}\mathbf{u}^{\epsilon'}\times \Delta J_{\epsilon'}\mathbf{u}^{\epsilon'}),\mathbf{u}^{\epsilon}-\mathbf{u}^{\epsilon'})_{\mathbb{L}^2}\\
&=-((J_{\epsilon}-J_{\epsilon'})(J_{\epsilon}\mathbf{u}^{\epsilon}\times \Delta J_{\epsilon}\mathbf{u}^{\epsilon})),\mathbf{u}^{\epsilon}-\mathbf{u}^{\epsilon'})_{\mathbb{L}^2}\\
&-(J_{\epsilon'}(J_{\epsilon}\mathbf{u}^{\epsilon}\times (J_{\epsilon}\Delta\mathbf{u}^{\epsilon}-J_{\epsilon'}\Delta\mathbf{u}^{\epsilon})),\mathbf{u}^{\epsilon}-\mathbf{u}^{\epsilon'})_{\mathbb{L}^2}\\
&-(J_{\epsilon'}(J_{\epsilon}\mathbf{u}^{\epsilon}\times (J_{\epsilon'}\Delta\mathbf{u}^{\epsilon}-J_{\epsilon'}\Delta\mathbf{u}^{\epsilon'})),\mathbf{u}^{\epsilon}-\mathbf{u}^{\epsilon'})_{\mathbb{L}^2}\\
&-(J_{\epsilon'}( (J_{\epsilon}\mathbf{u}^{\epsilon}-J_{\epsilon'}\mathbf{u}^{\epsilon})\times J_{\epsilon'}\Delta\mathbf{u}^{\epsilon'}),\mathbf{u}^{\epsilon}-\mathbf{u}^{\epsilon'})_{\mathbb{L}^2}\\
&-(J_{\epsilon'}((J_{\epsilon'}\mathbf{u}^{\epsilon}-J_{\epsilon'}\mathbf{u}^{\epsilon'})\times J_{\epsilon'}\Delta\mathbf{u}^{\epsilon'}),\mathbf{u}^{\epsilon}-\mathbf{u}^{\epsilon'})_{\mathbb{L}^2}\\
&\lesssim\max\{\epsilon,\epsilon'\}(\|\mathbf{u}^{\epsilon}\|_{\mathbb{H}^3}^2+\|\mathbf{u}^{\epsilon}\|_{\mathbb{H}^3}\|\mathbf{u}^{\epsilon'}\|_{\mathbb{H}^2})\|\mathbf{u}^{\epsilon}-\mathbf{u}^{\epsilon'}\|_{\mathbb{L}^2}\\
&+C_{\varepsilon}(\|\mathbf{u}^{\epsilon}\|_{\mathbb{H}^{2}}^2+\|\mathbf{u}^{\epsilon'}\|_{\mathbb{H}^{4}})\|\mathbf{u}^{\epsilon}-\mathbf{u}^{\epsilon'}\|_{\mathbb{L}^2}^2+\varepsilon\|J_{\epsilon'}\Delta(\mathbf{u}^{\epsilon}-\mathbf{u}^{\epsilon'})\|_{\mathbb{L}^2}^2.
\end{split}
\end{equation*}
Plugging the estimates for $K_1$-$K_6$ into \eqref{4-8} and choosing $\varepsilon$ small enough, we infer from Lemma \ref{lem3-2} that
\begin{equation}\label{4-9}
\begin{split}
&\frac{\mathrm{d}}{\mathrm{d}t}\|\mathbf{u}^{\epsilon}-\mathbf{u}^{\epsilon'}\|_{\mathbb{L}^2}\lesssim C_{\|\mathbf{u}_0\|_{\mathbb{H}^6},T'}\max\{\epsilon,\epsilon'\}+C_{\|\mathbf{u}_0\|_{\mathbb{H}^6},T'}\|\mathbf{u}^{\epsilon}-\mathbf{u}^{\epsilon'}\|_{\mathbb{L}^2}.
\end{split}
\end{equation}
Thus applying the Gronwall lemma to \eqref{4-9}, we obtain the result. The proof is thus completed.
\end{proof}
\begin{corollary}\label{cor2}
Let $\mathbf{u}_0 \in \mathbb{H}^k$.
\begin{itemize}
    \item If $k=6$, then the sequence $\mathbf{u}^{\epsilon}$ strongly converges to $\mathbf{u}$ in $\mathcal{C}([0,T'];\mathcal{C}^4(\mathbb{R}^3))$.
    \item If $k=6+4m$, $m \in \mathbb{N}^+$, then $\mathbf{u} \in \bigcap_{k=0}^{m+1} \mathcal{C}^k([0,T'];\mathcal{C}^{4(m+1-k)}(\mathbb{R}^3))$.
\end{itemize}
Moreover, there exists a maximal time of existence $T_{M}>0$ and a solution
$$
\mathbf{u} \in \bigcap_{k=0}^{m+1} \mathcal{C}^k([0,T_{M});\mathcal{C}^{4(m+1-k)}(\mathbb{R}^3))
$$
when $k=6+4m$, $m \in \mathbb{N}$. In particular, for any given $T>0$, if
$
\sup_{t \in [0,T]} \|\nabla \mathbf{u}(t)\|_{\mathbb{L}^2} < \infty,
$
then necessarily $T_M \geq T$ (the solution can be extended to $[0,T]$). Equivalently, if $T_M \leq T$, then
$
\lim_{t \nearrow T_M} \|\nabla \mathbf{u}(t)\|_{\mathbb{L}^2} = \infty.
$
\end{corollary}
\begin{proof}[\emph{\textbf{Proof}}]
As an immediate consequence of Lemma \ref{lem4-1}, we obtain the existence of a limit function $\mathbf{u}\in \mathcal{C}([0,T'];\mathbb{L}^2)$ satisfying the uniform convergence estimate
\begin{equation}\label{4-10}
\sup_{t\in[0,T']}\|\mathbf{u}^{\epsilon}-\mathbf{u}\|_{\mathbb{L}^2} \lesssim_{\|\mathbf{u}_0\|_{\mathbb{H}^k},T'} \epsilon.
\end{equation}
Furthermore, applying the Gagliardo-Nirenberg interpolation inequality yields the existence of a constant $C_k>0$ such that for any $f\in H^k(\mathbb{R}^3)$ and intermediate regularity index $0<k'<k$, we have
\begin{equation}\label{4-11}
\|f\|_{H^{k'}} \leq C_k\|f\|_{H^k}^{\frac{k'}{k}}\|f\|_{L^2}^{1-\frac{k'}{k}}.
\end{equation}
By selecting $\frac{11}{2}<k'<k$ and combining Lemma \ref{lem3-2} with estimates \eqref{4-10} and \eqref{4-11}, we derive the refined convergence rate
\begin{equation*}
\sup_{t\in[0,T']}\|\mathbf{u}^{\epsilon}-\mathbf{u}\|_{\mathbb{H}^{k'}} \lesssim_{\|\mathbf{u}_0\|_{\mathbb{H}^k},T'} \epsilon^{1-\frac{k'}{k}}.
\end{equation*}
The Sobolev embedding $H^{k'}(\mathbb{R}^3)\hookrightarrow \mathcal{C}^4(\mathbb{R}^3)$ for $k'>\frac{11}{2}$ implies the strong convergence
$$
\mathbf{u}^{\epsilon} \to \mathbf{u} \quad \text{in} \quad \mathcal{C}([0,T'];\mathcal{C}^4(\mathbb{R}^3))
$$
when $k=6$. Moreover, from the regularized equation \eqref{Mod-1}, we deduce that $\mathbf{u}_t^{\epsilon}\rightarrow\mathbf{u}_t$ strongly in $\mathcal{C}([0,T'];\mathcal{C}(\mathbb{R}^3))$, confirming that $\mathbf{u}$ constitutes a classical solution to the LLBar equation. We remark that $H^{s}(\mathbb{R}^3)$ forms a Banach algebra for $s>\frac{3}{2}$. Consequently, through a bootstrap argument leveraging the structural properties of the LLBar equation, we obtain the enhanced regularity
$
\mathbf{u}\in \bigcap_{k=0}^{m+1}\mathcal{C}^k([0,T'];\mathcal{C}^{4(m+1-k)}(\mathbb{R}^3))
$
for initial data $\mathbf{u}_0\in\mathbb{H}^k$ with $k=6+4m$, $m\in\mathbb{N}^+$. The blow-up criterion remains identical to that established in Corollary \ref{cor1}, completing the proof.
\end{proof}

\section{Proof of Theorem \ref{the1}}\label{sec5}

\subsection{Existence of global solutions}\label{sec5-1}
Based on Corollary \ref{cor1} in Section \ref{sec4}, we know that there exists a maximal time $T_M$ such that the equation
\begin{equation}\label{5-1}
\begin{split}
&\mathbf{u}_t=-\Delta^2\mathbf{u}-\Delta\mathbf{u}+2(1-|\mathbf{u}|^2)\mathbf{u}+2\Delta(|\mathbf{u}|^2\mathbf{u})-\mathbf{u}\times\Delta \mathbf{u}
\end{split}
\end{equation}
holds almost everywhere with respect to $(t,x)\in [0,T_M)\times\mathbb{R}^3$. We shall fully utilize the structural features of the LLBar equation itself to obtain the global solution. Specifically, taking the $\mathbb{L}^2$ inner product of both sides of equation \eqref{5-1} with $|\mathbf{u}|^2\mathbf{u}$, we have
\begin{equation}\label{5-2}
\begin{split}
&\frac{1}{2}\frac{\mathrm{d}}{\mathrm{d}t}\|\mathbf{u}\|_{\mathbb{L}^4}^4=(\mathbf{H}_{\textrm{eff}}-\Delta\mathbf{H}_{\textrm{eff}}-\mathbf{u}\times \mathbf{H}_{\textrm{eff}},2|\mathbf{u}|^2\mathbf{u})_{\mathbb{L}^2},
\end{split}
\end{equation}
where $\mathbf{H}_{\textrm{eff}}=\Delta \mathbf{u}+2(1-|\mathbf{u}|^2)\mathbf{u}$ comes from \eqref{sys0}. Moreover, taking the $\mathbb{L}^2$ inner product of both sides of equation \eqref{5-1} with $-\Delta \mathbf{u}$ and $\mathbf{u}$, respectively, we obtain
\begin{equation}\label{5-3}
\begin{split}
&\frac{1}{2}\frac{\mathrm{d}}{\mathrm{d}t}\|\nabla\mathbf{u}\|_{\mathbb{L}^2}^2=(\mathbf{H}_{\textrm{eff}}-\Delta\mathbf{H}_{\textrm{eff}}-\mathbf{u}\times \mathbf{H}_{\textrm{eff}},-\Delta\mathbf{u})_{\mathbb{L}^2},
\end{split}
\end{equation}
and
\begin{equation}\label{5-3}
\begin{split}
&\frac{\mathrm{d}}{\mathrm{d}t}\|\mathbf{u}\|_{\mathbb{L}^2}^2=(\mathbf{H}_{\textrm{eff}}-\Delta\mathbf{H}_{\textrm{eff}}-\mathbf{u}\times \mathbf{H}_{\textrm{eff}},2\mathbf{u})_{\mathbb{L}^2}.
\end{split}
\end{equation}
By combining \eqref{5-1}, \eqref{5-2}, and \eqref{5-3}, and noting the definition of $\mathbf{H}_{\textrm{eff}}$, it follows that
\begin{equation}\label{5-4}
\begin{split}
&\frac{\mathrm{d}}{\mathrm{d}t}(\frac{1}{2}\|\mathbf{u}\|_{\mathbb{L}^4}^4+\frac{1}{2}\|\nabla\mathbf{u}\|_{\mathbb{L}^2}^2-\|\mathbf{u}\|_{\mathbb{L}^2}^2)=(\mathbf{H}_{\textrm{eff}}-\Delta\mathbf{H}_{\textrm{eff}}-\mathbf{u}\times \mathbf{H}_{\textrm{eff}},-\mathbf{H}_{\textrm{eff}})_{\mathbb{L}^2}\\
&=-\|\mathbf{H}_{\textrm{eff}}\|_{H^1}^2\leq0.
\end{split}
\end{equation}
Based on \eqref{5-3} and a discussion similar to \eqref{lem1-2}, we readily obtain
\begin{equation*}
\begin{split}
&\frac{\mathrm{d}}{\mathrm{d}t}\|\mathbf{u}\|_{\mathbb{L}^2}^2+\|\Delta \mathbf{u}\|_{\mathbb{L}^2}^2\leq C\|\mathbf{u}\|_{\mathbb{L}^2}^2.
\end{split}
\end{equation*}
This combined with \eqref{5-4} implies that for any $t\in[0,T]$,
\begin{equation*}
\begin{split}
&\|\mathbf{u}(t)\|_{\mathbb{H}^1}^2+\int_0^t\|\mathbf{u}\|_{\mathbb{H}^3}^2\,\mathrm{d}s\leq C_{\|\mathbf{u}_0\|_{\mathbb{H}^1},T}.
\end{split}
\end{equation*}
Moreover, similar to the uniform bounded estimates in Lemma \ref{lem3-1} and Lemma \ref{lem3-2}, it is not difficult to prove that for any $t\in[0,T]$ and $\mathbf{u}_0\in \mathbb{H}^k$, $k\geq2$,
\begin{equation*}
\begin{split}
&\|\mathbf{u}(t)\|_{\mathbb{H}^k}^2+\int_0^t\|\mathbf{u}\|_{\mathbb{H}^{k+2}}^2\,\mathrm{d}s+\int_0^t\|\frac{\mathrm{d}}{\mathrm{d}t}\mathbf{u}\|_{\mathbb{H}^{k-2}}^2\,\mathrm{d}s\leq C_{\|\mathbf{u}_0\|_{\mathbb{H}^k},T}.
\end{split}
\end{equation*}
Thus, by using the blow-up criterion in Corollary \ref{cor1} and Corollary \ref{cor2}, we deduce that for any $T>0$, $\mathbf{u}\in \mathcal{C}([0,T];\mathbb{H}^2)\cap L^2(0,T;\mathbb{H}^4)$ when $\mathbf{u}_0\in \mathbb{H}^2$, and $\mathbf{u}\in \bigcap_{k=0}^{m+1}\mathcal{C}^k([0,T];\mathcal{C}^{4(m+1-k)}(\mathbb{R}^3))$ when $\mathbf{u}_0\in \mathbb{H}^{6+4m}$, where $m\in\mathbb{N}$. The global existence of strong, classical, and smooth solutions is thus established.

\subsection{Uniqueness of global solutions}\label{sec5-2}
The result on the uniqueness of the solution is as follows.
\begin{proposition}\label{pro2} Let $\mathbf{u}_1$ and $\mathbf{u}_2$ be two solutions to problem \eqref{sys1} with the same initial data $\mathbf{u}(0)\in \mathbb{H}^k$, $k\geq2$. Then, $\mathbf{u}_1\equiv\mathbf{u}_2$ in $\mathcal{C}([0,T];\mathbb{H}^k)$.
\end{proposition}
\begin{proof}[\emph{\textbf{Proof}}] Based on the existence proofs of solutions in Sections \ref{sec3} and \ref{sec4}, we know that $\mathbf{u}_1$ and $\mathbf{u}_2$ satisfy the following space-time regularity:
\begin{equation}\label{un1}
\begin{split}
&\mathbf{u}_i\in \mathcal{C}([0,T];\mathbb{H}^k)\cap L^2(0,T;\mathbb{H}^{k+2})\cap W^{1,2}(0,T;\mathbb{H}^{k-2}),~i=1,~2.
\end{split}
\end{equation}
Let $\mathbf{u}^*:=\mathbf{u}_1-\mathbf{u}_2$. Then $\mathbf{u}^*$ satisfies the following equation
\begin{equation*}
\begin{split}
&\frac{\mathrm{d}}{\mathrm{d}t}\mathbf{u}^*+\Delta^2\mathbf{u}^*\\
&=-\Delta\mathbf{u}^*+2\mathbf{u}^*-2(|\mathbf{u}_1|^2\mathbf{u}_1-|\mathbf{u}_2|^2\mathbf{u}_2)+2\Delta(|\mathbf{u}_1|^2\mathbf{u}_1-|\mathbf{u}_2|^2\mathbf{u}_2)-\mathbf{u}^*\times\Delta \mathbf{u}_1-\mathbf{u}_2\times\Delta \mathbf{u}^*,
\end{split}
\end{equation*}
with initial data $\mathbf{u}^*(0)=0$. Taking the $\mathbb{L}^2$ inner product of both sides of the above equation with $\Lambda^{2k}\mathbf{u}^*$, it follows that
\begin{equation}\label{un2}
\begin{split}
&\frac{1}{2}\frac{\mathrm{d}}{\mathrm{d}t}\|\Lambda^{k}\mathbf{u}^*\|_{\mathbb{L}^2}^2+\|\Delta\Lambda^{k}\mathbf{u}^*\|_{\mathbb{L}^2}^2\\
&=\|\nabla\Lambda^{k}\mathbf{u}^*\|_{\mathbb{L}^2}^2+2\|\Lambda^{k}\mathbf{u}^*\|_{\mathbb{L}^2}^2-2(\Lambda^{k}(|\mathbf{u}_1|^2\mathbf{u}_1-|\mathbf{u}_2|^2\mathbf{u}_2),\Lambda^{k}\mathbf{u}^*)_{\mathbb{L}^2}\\
&+2(\Lambda^{k}(|\mathbf{u}_1|^2\mathbf{u}_1-|\mathbf{u}_2|^2\mathbf{u}_2),\Delta\Lambda^{k}\mathbf{u}^*)_{\mathbb{L}^2}-(\Lambda^{k}(\mathbf{u}^*\times\Delta \mathbf{u}_1),\Lambda^{k}\mathbf{u}^*)_{\mathbb{L}^2}\\
&+(\Lambda^{k}(\mathbf{u}_2\times\nabla \mathbf{u}^*),\nabla\Lambda^{k}\mathbf{u}^*)_{\mathbb{L}^2}:=L_1+\cdots +L_6.
\end{split}
\end{equation}
By using integration by parts and Young's inequality, it follows that
\begin{equation*}
\begin{split}
&L_1+L_2\leq \varepsilon\|\Delta\Lambda^{k}\mathbf{u}^*\|_{\mathbb{L}^2}^2+C_{\varepsilon}\|\Lambda^{k}\mathbf{u}^*\|_{\mathbb{L}^2}^2.
\end{split}
\end{equation*}
Noting that $H^{k}(\mathbb{R}^3)$ is a Banach algebra and $\|\Lambda^k\cdot\|_{L^2}\asymp\|\cdot\|_{H^k}$, we have
\begin{equation*}
\begin{split}
&L_3\lesssim\||\mathbf{u}_1|^2\mathbf{u}_1-|\mathbf{u}_2|^2\mathbf{u}_2\|_{\mathbb{H}^k}\|\mathbf{u}^*\|_{\mathbb{H}^k}\lesssim(\|\mathbf{u}_1\|_{\mathbb{H}^k}^2+\|\mathbf{u}_2\|_{\mathbb{H}^k}^2)\|\mathbf{u}^*\|_{\mathbb{H}^k}^2.
\end{split}
\end{equation*}
Similarly, by Young's inequality, it follows that
\begin{equation*}
\begin{split}
&L_4\leq \varepsilon\|\Delta\Lambda^{k}\mathbf{u}^*\|_{\mathbb{L}^2}^2+C_{\varepsilon}\||\mathbf{u}_1|^2\mathbf{u}_1-|\mathbf{u}_2|^2\mathbf{u}_2\|_{\mathbb{H}^k}^2\\
&\leq\varepsilon\|\Delta\Lambda^{k}\mathbf{u}^*\|_{\mathbb{L}^2}^2+C_{\varepsilon}(\|\mathbf{u}_1\|_{\mathbb{H}^k}^4+\|\mathbf{u}_2\|_{\mathbb{H}^k}^4)\|\mathbf{u}^*\|_{\mathbb{H}^k}^2.
\end{split}
\end{equation*}
For $L_5$, it follows that
\begin{equation*}
\begin{split}
&L_5\lesssim\|\mathbf{u}_1\|_{\mathbb{H}^{k+2}}\|\mathbf{u}^*\|_{\mathbb{H}^k}^2.
\end{split}
\end{equation*}
For $L_6$, by using the fact that
\begin{equation*}
\begin{split}
&\Lambda^{k}(\mathbf{u}_2\times\nabla \mathbf{u}^*)\asymp\Lambda^{k}\mathbf{u}_2\times\nabla \mathbf{u}^*+\mathbf{u}_2\times (\nabla \Lambda^{k}\mathbf{u}^*)+C_k\sum_{i=1}^{k-1}\Lambda^{i}\mathbf{u}_2\times(\nabla\Lambda^{k-i}\mathbf{u}^*),
\end{split}
\end{equation*}
and the Sobolev embedding $H^2(\mathbb{R}^3)\hookrightarrow L^{\infty}(\mathbb{R}^3)$, it follows that
\begin{equation*}
\begin{split}
&L_6\lesssim|(\Lambda^{k}\mathbf{u}_2\times \nabla \mathbf{u}^*,\nabla\Lambda^{k} \mathbf{u}^*)_{\mathbb{L}^2}|+C_k\sum_{i=1}^{k-1}|(\Lambda^{i}\mathbf{u}_2\times(\nabla\Lambda^{k-i}\mathbf{u}^*),\nabla\Lambda^{k} \mathbf{u}^*)_{\mathbb{L}^2}|\\
&\lesssim\|\Lambda^{k}\mathbf{u}_2\|_{\mathbb{L}^{\infty}}\|\nabla \mathbf{u}^*\|_{\mathbb{L}^{2}}\|\nabla\Lambda^{k} \mathbf{u}^*\|_{\mathbb{L}^{2}}+C_k\sum_{i=1}^{k-1}\|\Lambda^{i}\mathbf{u}_2\|_{\mathbb{L}^{\infty}}\|\nabla\Lambda^{k-i}\mathbf{u}^*\|_{\mathbb{L}^2}\|\nabla\Lambda^{k} \mathbf{u}^*\|_{\mathbb{L}^2}\\
&\lesssim\varepsilon\|\nabla\Lambda^{k} \mathbf{u}^*\|_{\mathbb{L}^{2}}^2+C_{\varepsilon,k}\|\mathbf{u}_2\|_{\mathbb{H}^{k+2}}^2\|\mathbf{u}^*\|_{\mathbb{H}^{k}}^2\\
&\leq\varepsilon\|\Delta\Lambda^{k}\mathbf{u}^*\|_{\mathbb{L}^2}^2+C_{\varepsilon,k}(1+\|\mathbf{u}_2\|_{\mathbb{H}^{k+2}}^2)\|\mathbf{u}^*\|_{\mathbb{H}^{k}}^2.
\end{split}
\end{equation*}
Plugging the estimates for $L_1$-$L_6$ into \eqref{un2} and choosing $\varepsilon$ small enough, it follows that
\begin{equation}\label{un3}
\begin{split}
&\frac{\mathrm{d}}{\mathrm{d}t}\|\mathbf{u}^*\|_{\mathbb{H}^k}^2+\|\Delta\mathbf{u}^*\|_{\mathbb{H}^k}^2\lesssim_k(1+\|\mathbf{u}_1\|_{\mathbb{H}^k}^4+\|\mathbf{u}_2\|_{\mathbb{H}^k}^4+\|\mathbf{u}_1\|_{\mathbb{H}^{k+2}}^2+\|\mathbf{u}_2\|_{\mathbb{H}^{k+2}}^2)\|\mathbf{u}^*\|_{\mathbb{H}^{k}}^2.
\end{split}
\end{equation}
Noting \eqref{un1} and applying the Gronwall lemma to \eqref{un3}, we derive that $\mathbf{u}^*=0$ in $\mathcal{C}([0,T];\mathbb{H}^k)$. The proof is thus completed.
\end{proof}
The proof of Theorem \ref{the1} is thus complete by combining the global existence and Proposition \ref{pro2}.
\section*{Acknowledgements}
This work was partially supported by the National Natural Science Foundation of China (Grant No. 12231008).

\bibliographystyle{plain}%
\bibliography{3DLLBar}

\end{document}